\documentclass[12pt,oneside,reqno]{amsart}

\usepackage{amsmath,amssymb,amsthm,textcomp}
\usepackage{amsfonts,graphicx}
\usepackage[mathscr]{eucal}
\usepackage{mathtools}
\usepackage{color}
\usepackage{csquotes}
\usepackage[backend=bibtex,%
firstinits=true,%
doi=false,%
isbn=true,%
url=false,%
maxnames=99]{biblatex}%

\AtEveryBibitem{\clearfield{issn}}
\AtEveryCitekey{\clearfield{issn}}
\numberwithin{equation}{section}
\DeclareNameAlias{sortname}{last-first}
\theoremstyle{definition}

\numberwithin{equation}{section}

\newcommand{\ncom}{\newcommand}

\ncom{\beq}{\begin{equation}}
\ncom{\eeq}{\end{equation}}
\ncom{\bea}{\begin{eqnarray*}}
\ncom{\eea}{\end{eqnarray*}}
\ncom{\beqa}{\begin{eqnarray}}
\ncom{\eeqa}{\end{eqnarray}}
\ncom{\nno}{\nonumber}
\ncom{\non}{\nonumber}
\ncom{\ds}{\displaystyle}
\ncom{\half}{\frac{1}{2}}
\ncom{\mbx}{\makebox{.25cm}}
\ncom{\hs}{\mbox{\hspace{.25cm}}}
\ncom{\rar}{\rightarrow}
\ncom{\Rar}{\Rightarrow}
\ncom{\noin}{\noindent}
\ncom{\bc}{\begin{center}}
\ncom{\ec}{\end{center}}
\ncom{\sz}{\scriptsize}
\ncom{\rf}{\ref}
\ncom{\s}{\sqrt{2}}
\ncom{\sgm}{\sigma}
\ncom{\Sgm}{\Sigma}
\ncom{\psgm}{\sigma^{\prime}}
\ncom{\dt}{\delta}
\ncom{\Dt}{\Delta}
\ncom{\lmd}{\lambda}
\ncom{\Lmd}{\Lambda}
\ncom{\Th}{\Theta}
\ncom{\e}{\eta}
\ncom{\eps}{\epsilon}
\ncom{\pcc}{\stackrel{P}{>}}
\ncom{\lp}{\stackrel{L_{p}}{>}}
\ncom{\dist}{{\rm\,dist}}
\ncom{\sspan}{{\rm\,span}}
\ncom{\re}{{\rm Re\,}}
\ncom{\im}{{\rm Im\,}}
\ncom{\sgn}{{\rm sgn\,}}
\ncom{\ba}{\begin{array}}
\ncom{\ea}{\end{array}}
\ncom{\hone}{\mbox{\hspace{1em}}}
\ncom{\htwo}{\mbox{\hspace{2em}}}
\ncom{\hthree}{\mbox{\hspace{3em}}}
\ncom{\hfour}{\mbox{\hspace{4em}}}
\ncom{\vone}{\vskip 2ex}
\ncom{\vtwo}{\vskip 4ex}
\ncom{\vonee}{\vskip 1.5ex}
\ncom{\vthree}{\vskip 6ex}
\ncom{\vfour}{\vspace*{8ex}}
\ncom{\norm}{\|\;\;\|}
\ncom{\integ}[4]{\int_{#1}^{#2}\,{#3}\,d{#4}}
\ncom{\vspan}[1]{{{\rm\,span}\{ #1 \}}}
\ncom{\dm}[1]{ {\displaystyle{#1} } }
\ncom{\ri}[1]{{#1} \index{#1}}

\newtheorem{theorem}{\bf Theorem}[section]
\newtheorem{remark}{\bf Remark}[section]
\newtheorem{proposition}{Proposition}[section]
\newtheorem{lemma}{Lemma}[section]
\newtheorem{corollary}{Corollary}[section]

\newtheorem{definition}{Definition}[section]
\newtheoremstyle
    {remarkstyle}
    {}
    {11pt}
    {}
    {}
    {\bfseries}
    {:}
    {     }
    {\thmname{#1} \thmnumber{#2} }

\theoremstyle{remarkstyle}

\def\eps{\varepsilon}

\begin{document}
\title{Convoluted Fractional Poisson Process}
\author[Kuldeep Kumar Kataria]{K. K. Kataria}
\address{Kuldeep Kumar Kataria, Department of Mathematics,
	Indian Institute of Technology Bhilai, Raipur-492015, India.}
\email{kuldeepk@iitbhilai.ac.in}
\author[Mostafizar Khandakar]{M. Khandakar}
\address{Mostafizar Khandakar, Department of Mathematics,
	Indian Institute of Technology Bhilai, Raipur-492015, India.}
\email{mostafizark@iitbhilai.ac.in}
\date{May 21, 2020.}
\begin{abstract}
In this paper, we introduce and study a convoluted version of the time fractional Poisson process by taking the discrete convolution with respect to space variable in the system of fractional differential equations that governs its state probabilities. We call the introduced process as the convoluted fractional Poisson process (CFPP). The explicit expression for the Laplace transform of its state probabilities are obtained whose inversion yields its one-dimensional distribution. Some of its statistical properties such as probability generating function, moment generating function, moments {\it etc.} are obtained. A special case of CFPP, namely, the convoluted Poisson process (CPP) is studied and its time-changed subordination relationships with CFPP are discussed. It is shown that the CPP is a L\'evy process using which the long-range dependence property of CFPP is established. Moreover, we show that the increments of CFPP exhibits short-range dependence property.
\end{abstract}

\maketitle
\section{Introduction}
The Poisson process is a renewal process with exponentially distributed waiting times. This L\'evy process is often used to model the counting phenomenon. Empirically, it is observed that the process with heavy-tailed distributed waiting times offers a better model than the ones with light-tailed distributed waiting times. For this purpose several fractional generalizations of the homogeneous Poisson process are introduced and studied by researchers in the past two decades. These generalizations give rise to some interesting connections between the theory of stochastic subordination, fractional calculus and renewal processes. These fractional processes can be broadly categorized into two types: the time fractional types and the space fractional types. 

The time fractional versions of the Poisson process are obtained by replacing the time derivative in the governing difference-differential equations of the state probabilities of Poisson process by certain fractional derivatives. These include Riemann-Liouville fractional derivative (see Laskin (2003)), Caputo fractional derivative (see Beghin and Orsingher (2009)), Prabhakar derivative (see Polito and Scalas (2016)), Saigo fractional derivative (see Kataria and Vellaisamy (2017a)) {\it etc}. These time fractional models is further generalized to state-dependent fractional Poisson processes (see Garra {\it et al.} (2015)) and the mixed fractional Poisson process (see Beghin (2012) and Aletti {\it et al.} (2018)). The governing difference-differential equations of state-dependent fractional Poisson processes depend on the number of events that have occurred till any given time $t$. The properties related to the notion of long memory such as the long-range dependence (LRD) property and the short-range dependence (SRD) property are obtained for such fractional processes by Biard and Saussereau (2014), Maheshwari and Vellaisamy (2016), Kataria and Khandakar (2020). 

Orsingher and Polito (2012) introduced a space fractional version of the Poisson process, namely, the space fractional Poisson process (SFPP). It is characterized as a stochastic process obtained by time-changing the Poisson process by an independent stable subordinator. Orsingher and Toaldo (2015) studied a class of generalized space fractional Poisson processes associated with Bern\v stein functions. A particular choice of Bern\v stein function leads to a specific point process. Besides the SFPP, this class includes the relativistic Poisson process and the gamma-Poisson process as particular cases. Beghin and Vellaisamy (2018) introduced and study a process by time-changing the SFPP by a gamma subordinator. The jumps can take any positive value is a specific characteristics of these generalized space fractional processes.

The time fractional Poisson process (TFPP), denoted by $\{N^\alpha(t)\}_{t\geq0}$, $0<\alpha\leq 1$, is a time fractional version of the homogeneous Poisson process whose state probabilities $p^\alpha(n,t)=\mathrm{Pr}\{N^\alpha(t)=n\}$ satisfy (see Laskin (2003), Beghin and Orsingher (2009))
\begin{equation}\label{qqawq112}
\partial_t^\alpha p^\alpha(n,t)=-\lambda p^\alpha(n,t)+\lambda p^\alpha(n-1,t),\ \ n\geq 0,
\end{equation}
with $p^\alpha(-1,t)=0$, $t\geq 0$ and the initial conditions $p^\alpha(0,0)=1$ and $p^\alpha(n,0)=0$, $n\geq 1$. Here, $\lambda>0$ is the intensity parameter.

The derivative $\partial_t^\alpha$ involved in (\ref{qqawq112}) is the Dzhrbashyan--Caputo fractional derivative defined as
\begin{equation}\label{plm1}
\partial^{\alpha}_{t}f(t)\coloneqq\begin{cases}
\dfrac{1}{\Gamma \left( 1-\alpha \right)}\displaystyle\int_{0}^{t}(t-s)^{-\alpha}f^{\prime}(s)\,\mathrm{d}s,\ \ 0<\alpha<1,\vspace*{.2cm}\\
f^{\prime}(t), \ \ \alpha=1,
\end{cases}
\end{equation}
whose Laplace transform is given by (see Kilbas {\it et al.} (2006), Eq. (5.3.3))
\begin{equation}\label{lc}
\mathcal{L}\left(\partial^{\alpha}_{t}f(t);s\right)=s^{\alpha}\tilde{f}(s)-s^{\alpha-1}f(0),\ \ s>0.
\end{equation}
For $\alpha=1$, the TFPP reduces to Poisson process.

The TFPP is characterized as a time-changed Poisson process $\{N(t)\}_{t\geq0}$ by an inverse $\alpha$-stable subordinator $\{H^{\alpha}(t)\}_{t\geq0}$ (see Meerschaert {\it et al.} (2011)), that is,
\begin{equation}\label{1.1df}
N^\alpha(t)\overset{d}{=}N(H^\alpha(t)),
\end{equation}
where $\overset{d}{=}$ means equal in distribution.

In this paper, we introduce a counting process by varying intensity as a function of states and by taking discrete convolution in (\ref{qqawq112}). The discrete convolution used is defined in (\ref{def}). We call the introduced process as the convoluted fractional Poisson process (CFPP)  and denote it by $\{\mathcal{N}^{\alpha}_{c}(t)\}_{t\ge0}$, $0<\alpha\le1$. It is defined as the stochastic process whose state probabilities $p^\alpha_{c}(n,t)=\mathrm{Pr}\{\mathcal{N}^{\alpha}_{c}(t)=n\}$, satisfy
\begin{equation}\label{modellq}
\partial^{\alpha}_{t}p^\alpha_{c}(n,t)=-\lambda_{n}*p^\alpha_{c}(n,t)+\lambda_{n-1}*p^\alpha_{c}(n-1,t),\  \ \ n\ge0,
\end{equation}
with initial conditions
\begin{equation*}
p^\alpha_{c}(n,0)=\begin{cases}
1,\ \ n=0,\\
0,\ \ n\ge1,
\end{cases}
\end{equation*}
and $p^\alpha_{c}(-n,t)=0$ for all $n\ge1$, $t\ge0$. Also, $\{\lambda_{j},\ j\in\mathbb{Z}\}$ is a non-increasing sequence of intensity parameters such that $\lambda_{j}=0$ for all $j<0$, $\lambda_{0}>0$ and $\lambda_{j}\ge0$ for all $j>0$ with $\lim\limits_{j\to\infty}\lambda_{j+1}/\lambda_{j}<1$.

The Laplace transform of the state probabilities of CFPP is inverted to obtain its one-dimensional distribution in terms of Mittag-Leffler function, defined in (\ref{mit}), as
\begin{equation*}
p^\alpha_{c}(n,t)=\begin{cases}
E_{\alpha,1}(-\lambda_{0}t^{\alpha}),\ \ n=0,\vspace*{.2cm}\\
\displaystyle\sum_{k=1}^{n}\sum_{\Theta_{n}^{k}}k!\prod_{j=1}^{n}\frac{(\lambda_{j-1}-\lambda_{j})^{k_{j}}}{k_{j}!} t^{k\alpha}E_{\alpha,k\alpha+1}^{k+1}(-\lambda_{0}t^{\alpha}),\ \  n\ge1,
\end{cases}
\end{equation*}
where the sum is taken over the set $\Theta_{n}^{k}=\{(k_{1},k_{2},\dots,k_{n}):\sum_{j=1}^{n}k_{j}=k, \ \  \sum_{j=1}^{n}jk_{j}=n,\ k_{j}\in\mathbb{N}_0\}$. An alternate expression for $p^\alpha_{c}(n,t)$ is obtained where the sum is taken over a slightly simplified set. It is observed that the CFPP is not a renewal process. The TFPP can be obtained as particular case of the CFPP by taking $\lambda_{n}=0$ for all $n\ge1$. Further, $\alpha=1$ gives the homogeneous Poisson process.

The paper is organized as follows: In Section \ref{Section2}, we set some notations and give some preliminary results related to Mittag-Leffler function, its generalizations, Bell polynomials {\it etc}. In Section \ref{Section3}, we introduce the CFPP and obtain its state probabilities. It is shown that the CFPP is a limiting case of a fractional counting process introduced and studied by Di Crescenzo {\it et al.} (2016). Its probability generating function (pgf), factorial moments, moment generating function (mgf), moments including mean and variance are derived. Also, it is shown that the CFPP is a fractional compound Poisson process. In Section \ref{Section4}, we study a particular case of the CFPP, namely, the convoluted Poisson process (CPP). It is shown that the CPP is a L\'evy process. Some subordination results related to CPP, CFPP and inverse stable subordinator are obtained. In Section \ref{Section5}, we have shown that the CFPP exhibits the LRD property whereas its increments have the SRD property.
\section{Preliminaries}\label{Section2}
The set of integers is denoted by $\mathbb{Z}$, the set of positive integers by $\mathbb{N}$ and the set of non-negative integers by $\mathbb{N}_0$. The following definitions and known results related to discrete convolution, Bell polynomials, Mittag-Leffler function and its generalizations will be used.
\subsection{Discrete Convolution}
The discrete convolution of two real-valued functions $f$ and $g$ whose support is the set of integers is defined as (see Damelin and Miller (2012), p. 232)
\begin{equation}\label{def}
	(f*g)(n)\coloneqq\sum_{j=-\infty}^{\infty}f(j)g(n-j).
\end{equation}
Here, $\sum_{j=-\infty}^{\infty}|f(j)|<\infty$ and $\sum_{j=-\infty}^{\infty}|g(j)|<\infty$, that is, $f,g\in\ell^1$ space.
\subsection{Bell polynomials}
The ordinary Bell polynomials $\hat{B}_{n,k}$ in $n-k+1$ variables is defined as
\begin{equation*}
\hat{B}_{n,k}(u_{1},u_{2},\dots,u_{n-k+1})\coloneqq\sum_{\Lambda_{n}^{k}}k!\prod_{j=1}^{n-k+1}\frac{u_{j}^{k_{j}}}{k_{j}!},
\end{equation*}
where
\begin{equation}\label{lambdaee11}
\Lambda_{n}^{k}=\left\{(k_{1},k_{2},\dots,k_{n-k+1}):\sum_{j=1}^{n-k+1}k_{j}=k, \ \  \sum_{j=1}^{n-k+1}jk_{j}=n, \ k_{j}\in\mathbb{N}_0\right\}.
\end{equation}
The following results holds (see Comtet 1974, pp. 133-137):
\begin{equation}\label{fm1}
\exp\left(x\sum_{j=1}^{\infty}u_{j}t^{j}\right)=1+\sum_{n=1}^{\infty}t^{n}\sum_{k=1}^{n}\hat{B}_{n,k}(u_{1},u_{2},\dots,u_{n-k+1})\frac{x^{k}}{k!}
\end{equation}
and 
\begin{equation}\label{fm2}
\left(\sum_{j=1}^{\infty}u_{j}t^{j}\right)^{k}=\sum_{n=k}^{\infty}\hat{B}_{n,k}(u_{1},u_{2},\dots,u_{n-k+1})t^{n}.
\end{equation}
\subsection{Mittag-Leffler function and its generalizations}
The Mellin-Barnes representation of the exponential function is given by (see Paris and Kaminski (2001), Eq. (3.3.2))
\begin{equation}\label{me}
e^{x}=\frac{1}{2\pi i}\int_{c-i\infty}^{c+i\infty}\Gamma (z)(-x)^{-z}\mathrm{d}z, \ \ x\neq0,
\end{equation}
where $i=\sqrt{-1}$.

The one-parameter Mittag-Leffler function is a generalization of the exponential function. It is defined as (see Mathai and Haubold (2008)) 
\begin{equation*}
	E_{\alpha, 1}(x)\coloneqq\sum_{k=0}^{\infty} \frac{x^{k}}{\Gamma(k\alpha+1)},\ \ x\in\mathbb{R},
\end{equation*}
where $\alpha>0$. For $\alpha=1$, it reduces to the exponential function. It is further generalized to two-parameter and three-parameter Mittag-Leffler functions. 

The three-parameter Mittag-Leffler function is defined as
\begin{equation}\label{mit}
	E_{\alpha, \beta}^{\gamma}(x)\coloneqq\frac{1}{\Gamma(\gamma)}\sum_{k=0}^{\infty} \frac{\Gamma(\gamma+k)x^{k}}{k!\Gamma(k\alpha+\beta)},\ \ x\in\mathbb{R},
\end{equation}
where $\alpha>0$, $\beta>0$ and $\gamma>0$. 

For $x\neq0$, its Mellin-Barnes representation is given by (see Mathai and Haubold (2008), Eq. (2.3.5))
\begin{equation}\label{m3}
E_{\alpha,\beta}^{\gamma}(x)=\frac{1}{2\pi i\Gamma(\gamma)}\int_{c-i\infty}^{c+i\infty}\frac{\Gamma (z)\Gamma(\gamma-z)}{\Gamma(\beta-\alpha z)}(-x)^{-z}\mathrm{d}z,
\end{equation}
where $0<c<\gamma$. For $\gamma=1$, it reduces to two-parameter Mittag-Leffler function. Further, $\beta=\gamma=1$ reduce it to one-parameter Mittag-Leffler function. Note that (\ref{m3}) reduces to (\ref{me}) for $\alpha=\beta=\gamma=1$.

Let $\alpha>0$, $\beta>0$, $t>0$ and $x,y$ be any two reals. The Laplace transform of the function $t^{\beta-1}E^{\gamma}_{\alpha,\beta}(\lambda t^{\alpha})$ is given by  (see Kilbas {\it et al.} (2006), Eq. (1.9.13)):
\begin{equation}\label{mi}
	\mathcal{L}\{t^{\beta-1}E^{\gamma}_{\alpha,\beta}(xt^{\alpha});s\}=\frac{s^{\alpha\gamma-\beta}}{(s^{\alpha}-x)^{\gamma}},\ s>|x|^{1/\alpha}.
\end{equation} 

The following result holds for three-parameter Mittag-Leffler function (see Oliveira {\it et al.} (2016), Theorem 3.2):
\begin{equation}\label{formula}
	\sum_{k=0}^{\infty} (yt^{\alpha})^{k}E_{\alpha,k\alpha+\beta}^{k+1}(xt^{\alpha})=E_{\alpha,\beta}((x+y)t^{\alpha}).
\end{equation}

Let $E_{\alpha, \beta}^{(n)}(.)$ denote the $n$th derivative of two-parameter Mittag-Leffler function. Then, (see Kataria and Vellaisamy (2019), Eq. (7)) :
\begin{equation}\label{re}
	E_{\alpha, \beta}^{(n)}(x)=n!E_{\alpha, n\alpha+\beta}^{n+1}(x),\ \ n\ge0.
\end{equation}

\section{Convoluted Fractional Poisson Process}\label{Section3}
Here, we introduce and study a counting process by varying intensity as a function of states and by taking discrete convolution in the governing difference-differential equation (\ref{qqawq112}) of the TFPP. The introduced process is called the convoluted fractional Poisson process (CFPP)  which we denote by $\{\mathcal{N}^{\alpha}_{c}(t)\}_{t\ge0}$, $0<\alpha\le1$. We define it as the stochastic process whose state probabilities $p^\alpha_{c}(n,t)=\mathrm{Pr}\{\mathcal{N}^{\alpha}_{c}(t)=n\}$, satisfy
\begin{equation}\label{con}
\partial^{\alpha}_{t}p^\alpha_{c}(n,t)=-\lambda_{n}*p^\alpha_{c}(n,t)+\lambda_{n-1}*p^\alpha_{c}(n-1,t),\  \ \ n\ge0,
\end{equation}
with initial conditions
\begin{equation*}
p^\alpha_{c}(n,0)=\begin{cases}
1,\ \ n=0,\\
0,\ \ n\ge1,
\end{cases}
\end{equation*}
and $p^\alpha_{c}(-n,t)=0$ for all $n\ge1$, $t\ge0$. 

Also, $\{\lambda_{j},\ j\in\mathbb{Z}\}$ is a non-increasing sequence of intensity parameters such that $\lambda_{j}=0$ for all $j<0$, $\lambda_{0}>0$ and $\lambda_{j}\ge0$ for all $j>0$ with $\lim\limits_{j\to\infty}\lambda_{j+1}/\lambda_{j}<1$. Thus, it follows that 
\begin{equation}\label{asdesa11}
\sum_{j=0}^{\infty}(\lambda_{j-1}-\lambda_{j})=0,
\end{equation}
as $\sum_{j=0}^{\infty}\lambda_{j}<\infty$ implies $\lambda_{j}\to0$ as $j\to\infty$.

Using (\ref{def}), the system of fractional differential equations (\ref{con}) can be rewritten as
\begin{align}\label{model}
\partial^{\alpha}_{t}p^\alpha_{c}(n,t)&=-\sum_{j=0}^{n}\lambda_{j}p^\alpha_{c}(n-j,t)+\sum_{j=0}^{n-1}\lambda_{j}p^\alpha_{c}(n-j-1,t)\nonumber\\
&=-\sum_{j=0}^{n}\lambda_{j}p^\alpha_{c}(n-j,t)+\sum_{j=1}^{n}\lambda_{j-1}p^\alpha_{c}(n-j,t)\nonumber\\
&=-\lambda_{0}p^\alpha_{c}(n,t)+\sum_{j=1}^{n}(\lambda_{j-1}-\lambda_{j})p^\alpha_{c}(n-j,t),\ \ n\ge0.
\end{align}
Note that for $\lambda_{n}=0$ for all $n\ge1$, the CFPP reduces to TFPP with intensity parameter $\lambda_{0}>0$.
\begin{remark}
Di Crescenzo {\it et al.} (2016) studied a fractional counting process $\{M^\alpha(t)\}_{t\geq0}$, $0<\alpha\leq 1$, which performs $k$ kinds of jumps of amplitude $1,2,\dots,k$ with positive rates $\Lambda_{1},\Lambda_{2},\ldots,\Lambda_{k}$ where $k\in\mathbb{N}$ is fixed. Its state probabilities $q^\alpha(n,t)=\mathrm{Pr}\{M^\alpha(t)=n\}$ satisfy (see Di Crescenzo {\it et al.} (2016), Eq. (2.3))
	\begin{equation}\label{cre}
		\partial^{\alpha}_{t}q^\alpha(n,t)=-(\Lambda_{1}+\Lambda_{2}+\dots+\Lambda_{k})q^\alpha(n,t)+	\sum_{j=1}^{\min\{n,k\}}\Lambda_{j}q^\alpha(n-j,t),\ \ n\ge0,
	\end{equation}
with
	\begin{equation*}
		q^\alpha(n,0)=\begin{cases}
			1,\ \ n=0,\\
			0,\ \ n\ge1.
		\end{cases}
	\end{equation*}
For $k=1$, the system of fractional differential equation (\ref{cre}) reduces to (\ref{qqawq112}). Thus, the TFPP follows as a particular case of $\{M^\alpha(t)\}_{t\geq0}$. It is important to note that the CFPP is not a particular case of $\{M^\alpha(t)\}_{t\geq0}$ for any choice of $k\in\mathbb{N}$. However, if we choose $\Lambda_{j}=\lambda_{j-1}-\lambda_{j}$, for all $j\ge1$, then $\Lambda_{1}+\Lambda_{2}+\dots+\Lambda_{k}=\lambda_{0}-\lambda_{k}$. As $\sum_{k=0}^{\infty}\lambda_{k}<\infty$ implies $\lambda_{k}\to0$ as $k\to\infty$, the system (\ref{cre}) reduce to $(\ref{model})$. Thus, the CFPP is obtained as a limiting process of $\{M^\alpha(t)\}_{t\geq0}$ by letting $k\to\infty$. 
\end{remark}
The following result gives the Laplace transform of the state probabilities of CFPP.
\begin{proposition}
The Laplace transform of the state probabilities $\tilde{p}^\alpha_{c}(n,s)$, $s>0$, of CFPP is given by
	\begin{equation}\label{lap}
\tilde{p}^\alpha_{c}(n,s)=\begin{cases}
\dfrac{s^{\alpha-1}}{s^{\alpha}+\lambda_{0}},\ \ n=0,\vspace*{.2cm}\\
\displaystyle\sum_{k=1}^{n}\sum_{\Theta_{n}^{k}}k!\prod_{j=1}^{n}\frac{(\lambda_{j-1}-\lambda_{j})^{k_{j}}}{k_{j}!}\frac{s^{\alpha-1}}{(s^{\alpha}+\lambda_{0})^{k+1}},\ \ n\ge1,
\end{cases}
\end{equation}
where $\Theta_{n}^{k}=\{(k_{1},k_{2},\dots,k_{n}):\sum_{j=1}^{n}k_{j}=k,\ \sum_{j=1}^{n}jk_{j}=n,\ k_{j}\in\mathbb{N}_0\}$.
\end{proposition}
\begin{proof}	
On applying the Laplace transform in (\ref{model}) and using (\ref{lc}), we get
	\begin{equation*}
	s^{\alpha}\tilde{p}^\alpha_{c}(n,s)-s^{\alpha-1}p^\alpha_{c}(n,0)=-\lambda_{0}\tilde{p}^\alpha_{c}(n,s)+\sum_{m=1}^{n}(\lambda_{m-1}-\lambda_{m})\tilde{p}^\alpha_{c}(n-m,s).
	\end{equation*}
	Thus,
	\begin{equation}\label{pns}
	\tilde{p}^\alpha_{c}(n,s)=(s^{\alpha}+\lambda_{0})^{-1}\left(\sum_{m=1}^{n}(\lambda_{m-1}-\lambda_{m})\tilde{p}^\alpha_{c}(n-m,s)+s^{\alpha-1}p^\alpha_{c}(n,0)\right).
	\end{equation}
Put $n=0$ in the above equation and use the initial conditions given in (\ref{con}) to obtain
\begin{equation}\label{l0}
\tilde{p}^\alpha_{c}(0,s)=\dfrac{s^{\alpha-1}}{s^{\alpha}+\lambda_{0}}.
\end{equation}	
So, the result holds for $n=0$. Next, we put $n=1$ in (\ref{pns}) to get
\begin{equation*}
\tilde{p}^\alpha_{c}(1,s)=\dfrac{(\lambda_{0}-\lambda_{1})\tilde{p}(0,s)}{s^{\alpha}+\lambda_{0}}=\dfrac{(\lambda_{0}-\lambda_{1})s^{\alpha-1}}{(s^{\alpha}+\lambda_{0})^{2}},
\end{equation*}
and the result holds for $n=1$. Now put $n=2$ in (\ref{pns}) to get
\begin{equation*}
\tilde{p}^\alpha_{c}(2,s)=\dfrac{(\lambda_{0}-\lambda_{1})^{2}s^{\alpha-1}}{(s^{\alpha}+\lambda_{0})^{3}}+\dfrac{(\lambda_{1}-\lambda_{2})s^{\alpha-1}}{(s^{\alpha}+\lambda_{0})^{2}}.
\end{equation*}
Substituting $n=3$ in (\ref{pns}), we get
\begin{equation*}
\tilde{p}^\alpha_{c}(3,s)=\dfrac{(\lambda_{0}-\lambda_{1})^{3}s^{\alpha-1}}{(s^{\alpha}+\lambda_{0})^{4}}+\dfrac{2(\lambda_{0}-\lambda_{1})(\lambda_{1}-\lambda_{2})s^{\alpha-1}}{(s^{\alpha}+\lambda_{0})^{3}}+\dfrac{(\lambda_{2}-\lambda_{3})s^{\alpha-1}}{(s^{\alpha}+\lambda_{0})^{2}}.
\end{equation*}
Now put $n=4$ in (\ref{pns}), we get
\begin{align*}
\tilde{p}^\alpha_{c}(4,s)&=\dfrac{(\lambda_{0}-\lambda_{1})^{4}s^{\alpha-1}}{(s^{\alpha}+\lambda_{0})^{5}}+\dfrac{3(\lambda_{0}-\lambda_{1})^{2}(\lambda_{1}-\lambda_{2})s^{\alpha-1}}{(s^{\alpha}+\lambda_{0})^{4}}\\
&\ \ +\dfrac{\left(2(\lambda_{0}-\lambda_{1})(\lambda_{2}-\lambda_{3})+(\lambda_{1}-\lambda_{2})^{2}\right)s^{\alpha-1}}{(s^{\alpha}+\lambda_{0})^{3}}        +\dfrac{(\lambda_{3}-\lambda_{4})s^{\alpha-1}}{(s^{\alpha}+\lambda_{0})^{2}}.
\end{align*}
Equivalently, 
\begin{equation*}
\tilde{p}^\alpha_{c}(4,s)=\sum_{k=1}^{4}\sum_{\Theta_{4}^{k}}k!\prod_{j=1}^{4}\frac{(\lambda_{j-1}-\lambda_{j})^{k_{j}}}{k_{j}!}\frac{s^{\alpha-1}}{(s^{\alpha}+\lambda_{0})^{k+1}} ,
\end{equation*}
where $\Theta_{4}^{k}=\{(k_{1},k_{2},k_{3},k_{4}):\ k_{1}+k_{2}+k_{3}+k_{4}=k,\ k_{1}+2k_{2}+3k_{3}+4k_{4}=4,\ k_{i}\in\mathbb{N}_0\}$.

 Assume the result (\ref{lap}) holds for $n=l$. From  (\ref{pns}), we have
\begin{align*}
\tilde{p}^\alpha_{c}(l+1,s)&=(s^{\alpha}+\lambda_{0})^{-1}\left(\sum_{m=1}^{l+1}(\lambda_{m-1}-\lambda_{m})\tilde{p}^\alpha_{c}(l+1-m,s)\right)\\
&=(s^{\alpha}+\lambda_{0})^{-1}\left(\sum_{m=1}^{l}(\lambda_{m-1}-\lambda_{m})\tilde{p}^\alpha_{c}(l+1-m,s)\right)+\frac{(\lambda_{l}-\lambda_{l+1})\tilde{p}^\alpha_{c}(0,s)}{(s^{\alpha}+\lambda_{0})}\\
&=\sum_{m=1}^{l}(\lambda_{m-1}-\lambda_{m})\sum_{k=1}^{l+1-m}\sum_{\Theta_{l+1-m}^{k}}k!\prod_{j=1}^{l+1-m}\frac{(\lambda_{j-1}-\lambda_{j})^{k_{j}}s^{\alpha-1}}{k_{j}!(s^{\alpha}+\lambda_{0})^{k+2}}+\frac{(\lambda_{l}-\lambda_{l+1})s^{\alpha-1}}{(s^{\alpha}+\lambda_{0})^{2}}\\
&=\sum_{k=1}^{l+1}\sum_{\Theta_{l+1}^{k}}k!\prod_{j=1}^{l+1}\frac{(\lambda_{j-1}-\lambda_{j})^{k_{j}}}{k_{j}!}\frac{s^{\alpha-1}}{(s^{\alpha}+\lambda_{0})^{k+1}}.
\end{align*}
Using the method of mathematical induction, the result (\ref{lap}) holds true for all $n\ge0$.
\end{proof}
The above result can be written in a different form by using the following result due to Kataria and Vellaisamy (2017b).
\begin{lemma}\label{ll1}
	Let $e^n_j$ denotes the $n$-tuple vector with unity at the $j$-th place and zero elsewhere. Then
	\begin{equation*}\label{2.0}
	\Theta^k_n=\left\{\sum_{j=1}^{n-k+1}k_je^n_j:\left(k_1,k_2,\ldots,k_{n-k+1}\right)\in\Lambda^k_n\right\},
	\end{equation*}
	where $\Lambda_{n}^{k}$ is given in (\ref{lambdaee11}).
\end{lemma}
Using the above result, an equivalent expression for $\tilde{p}^\alpha_{c}(n,s)$ is given by
	\begin{equation*}
	\tilde{p}^\alpha_{c}(n,s)=\begin{cases}
	\dfrac{s^{\alpha-1}}{s^{\alpha}+\lambda_{0}},\ \ n=0,\vspace*{.2cm}\\
\displaystyle\sum_{k=1}^{n}\sum_{\Lambda_{n}^{k}}k!\prod_{j=1}^{n-k+1}\frac{(\lambda_{j-1}-\lambda_{j})^{k_{j}}}{k_{j}!}\frac{s^{\alpha-1}}{(s^{\alpha}+\lambda_{0})^{k+1}},\ \ n\ge1.
	\end{cases}
	\end{equation*}
\begin{theorem}
	The one-dimensional distribution of the CFPP is given by
	\begin{equation}\label{dist}
	p^\alpha_{c}(n,t)=\begin{cases}
	E_{\alpha,1}(-\lambda_{0}t^{\alpha}),\ \ n=0,\vspace*{.2cm}\\
	\displaystyle\sum_{k=1}^{n}\sum_{\Theta_{n}^{k}}k!\prod_{j=1}^{n}\frac{(\lambda_{j-1}-\lambda_{j})^{k_{j}}}{k_{j}!} t^{k\alpha}E_{\alpha,k\alpha+1}^{k+1}(-\lambda_{0}t^{\alpha}),\ \  n\ge1,
	\end{cases}
	\end{equation}
	where $\Theta_{n}^{k}=\{(k_{1},k_{2},\dots,k_{n}):\sum_{j=1}^{n}k_{j}=k, \ \  \sum_{j=1}^{n}jk_{j}=n,\ k_{j}\in\mathbb{N}_0\}$. 
\end{theorem}
\begin{proof}
	Taking inverse Laplace transform in (\ref{lap}), we get
		\begin{equation*}
	\mathcal{L}^{-1}\left(\tilde{p}^\alpha_{c}(n,s);t\right)=\begin{cases}
	\mathcal{L}^{-1}\left(\dfrac{s^{\alpha-1}}{s^{\alpha}+\lambda_{0}};t\right),\ \ n=0,\vspace*{.2cm}\\
	\mathcal{L}^{-1}\left(\displaystyle\sum_{k=1}^{n}\sum_{\Theta_{n}^{k}}k!\prod_{j=1}^{n}\frac{(\lambda_{j-1}-\lambda_{j})^{k_{j}}}{k_{j}!}\frac{s^{\alpha-1}}{(s^{\alpha}+\lambda_{0})^{k+1}};t\right),\ \ n\ge1.
	\end{cases}
	\end{equation*}
	Using (\ref{mi}), the above equation reduces to (\ref{dist}).
\end{proof}
The distribution of TFPP follows as a particular case of CFPP.
\begin{corollary}
Let $\lambda_{0}=\lambda$ and $\lambda_{n}=0$ for all $n\ge1$. Then,
	\begin{equation*}\label{dist1qa}
		p^\alpha(n,t)=\frac{(\lambda t^{\alpha})^{n}}{n!}\sum_{j=0}^{\infty}\frac{(j+n)!}{j!}\frac{(-\lambda t^{\alpha})^{j}}{\Gamma((j+n)\alpha+1)},\ \ n\geq0,
	\end{equation*}
which is the distribution of TFPP. 
\end{corollary}
\begin{proof}
On substituting $\lambda_{0}=\lambda$ and $\lambda_{n}=0$ for $n\ge1$ in (\ref{dist}), we get
\begin{equation*}
p^{\alpha}(n,t)=(\lambda t^{\alpha})^nE_{\alpha,n\alpha+1}^{n+1}(-\lambda t^{\alpha}),\ \  n\ge0.
\end{equation*}
On using (\ref{mit}), the result follows.
\end{proof}
Using Lemma \ref{ll1}, we can rewrite (\ref{dist}) as
\begin{equation}\label{dssdew1}
p^\alpha_{c}(n,t)=\begin{cases}
E_{\alpha,1}(-\lambda_{0}t^{\alpha}),\ \ n=0,\vspace*{.2cm}\\
\displaystyle\sum_{k=1}^{n}\sum_{\Lambda_{n}^{k}}k!\prod_{j=1}^{n-k+1}\frac{(\lambda_{j-1}-\lambda_{j})^{k_{j}}}{k_{j}!} t^{k\alpha}E_{\alpha,k\alpha+1}^{k+1}(-\lambda_{0}t^{\alpha}),\ \ n\ge1.
\end{cases}
\end{equation}
Note that,
\begin{align*}
\sum_{n=0}^{\infty}p^\alpha_{c}(n,t)&=p^\alpha_{c}(0,t)+\sum_{n=1}^{\infty}p^\alpha_{c}(n,t)\\
&=E_{\alpha,1}(-\lambda_{0}t^{\alpha})+\sum_{n=1}^{\infty}\sum_{k=1}^{n}\sum_{\Lambda_{n}^{k}}k!\prod_{j=1}^{n-k+1}\frac{(\lambda_{j-1}-\lambda_{j})^{k_{j}}}{k_{j}!} t^{k\alpha}E_{\alpha,k\alpha+1}^{k+1}(-\lambda_{0}t^{\alpha})\\
&=E_{\alpha,1}(-\lambda_{0}t^{\alpha})+\sum_{k=1}^{\infty}t^{k\alpha}E_{\alpha,k\alpha+1}^{k+1}(-\lambda_{0}t^{\alpha})\sum_{n=k}^{\infty}\sum_{\Lambda_{n}^{k}}k!\prod_{j=1}^{n-k+1}\frac{(\lambda_{j-1}-\lambda_{j})^{k_{j}}}{k_{j}!}\\
&=E_{\alpha,1}(-\lambda_{0}t^{\alpha})+\sum_{k=1}^{\infty}\left(\sum_{j=1}^{\infty}(\lambda_{j-1}-\lambda_{j})\right)^{k} t^{k\alpha}E_{\alpha,k\alpha+1}^{k+1}(-\lambda_{0}t^{\alpha}),\ \ \text{(using\ (\ref{fm2}))}\\
&=E_{\alpha,1}(-\lambda_{0}t^{\alpha})+\sum_{k=1}^{\infty} (\lambda_{0}t^{\alpha})^{k}E_{\alpha,k\alpha+1}^{k+1}(-\lambda_{0}t^{\alpha})\\
&=\sum_{k=0}^{\infty} (\lambda_{0}t^{\alpha})^{k}E_{\alpha,k\alpha+1}^{k+1}(-\lambda_{0}t^{\alpha})=1,
\end{align*}	
where in the last step we have used (\ref{formula}). Thus, it follows that (\ref{dist}) is a valid distribution.
\begin{remark}
	Let the random variable $W^{\alpha}_{c}$ be the waiting time of the first convoluted fractional Poisson event. Then, the distribution of $W^{\alpha}_{c}$ is given by
	\begin{equation*}
	\mathrm{Pr}\{W^{\alpha}_{c}>t\}=\mathrm{Pr}\{\mathcal{N}^{\alpha}_{c}(t)=0\}=E_{\alpha,1}(-\lambda_{0}t^{\alpha}), \ t>0,
	\end{equation*}
	which coincides with the first waiting time of TFPP (see Beghin and Orsingher (2009)).	However, the one-dimensional distributions of TFPP and CFPP differ. Thus, the fact that the TFPP is a renewal process (see Meerschaert {\it et al.} (2011)) implies that the CFPP is not a renewal process.	
\end{remark}
The next result gives the probability generating function (pgf) of CFPP.
\begin{proposition}
	The pgf $G^{\alpha}_{c}(u,t)=\mathbb{E}(u^{\mathcal{N}^{\alpha}_{c}(t)})$ of CFPP is given by
\begin{equation}\label{pgfa}
G^{\alpha}_{c}(u,t)=E_{\alpha,1}\left(\sum_{j=0}^{\infty}u^{j}(\lambda_{j-1}-\lambda_{j})t^{\alpha}\right),\ \ |u|\le1.
\end{equation}
\end{proposition}
\begin{proof}
The Laplace transform of the pgf of CFPP can be obtained as follows:
\begin{align}
\tilde{G}^{\alpha}_{c}(u,s)&=\int_{0}^{\infty}e^{-st}G^{\alpha}_{c}(u,t)\mathrm{d}t,\ \ s>0 \nonumber\\
&=\int_{0}^{\infty}e^{-st}\sum_{n=0}^{\infty}u^{n}p^\alpha_{c}(n,t)\mathrm{d}t\nonumber\\
&=\int_{0}^{\infty}e^{-st}\left(E_{\alpha,1}(-\lambda_{0}t^{\alpha})+\sum_{n=1}^{\infty}u^{n}\sum_{k=1}^{n}\sum_{\Lambda_{n}^{k}}k!\prod_{j=1}^{n-k+1}\frac{(\lambda_{j-1}-\lambda_{j})^{k_{j}}}{k_{j}!} t^{k\alpha}E_{\alpha,k\alpha+1}^{k+1}(-\lambda_{0}t^{\alpha})\right)\mathrm{d}t\nonumber\\
&=\dfrac{s^{\alpha-1}}{s^{\alpha}+\lambda_{0}}+\sum_{n=1}^{\infty}u^{n}\sum_{k=1}^{n}\sum_{\Lambda_{n}^{k}}k!\prod_{j=1}^{n-k+1}\frac{(\lambda_{j-1}-\lambda_{j})^{k_{j}}}{k_{j}!}\frac{s^{\alpha-1}}{(s^{\alpha}+\lambda_{0})^{k+1}},\ \ (\mathrm{using}\ (\ref{mi}))\label{mjhgqq1}\\
&=\dfrac{s^{\alpha-1}}{s^{\alpha}+\lambda_{0}}\left(1+\sum_{k=1}^{\infty}\frac{1}{(s^{\alpha}+\lambda_{0})^{k}}\sum_{n=k}^{\infty}\sum_{\Lambda_{n}^{k}}k!\prod_{j=1}^{n-k+1}\frac{(\lambda_{j-1}-\lambda_{j})^{k_{j}}}{k_{j}!} u^{n}\right)\nonumber\\
&=\dfrac{s^{\alpha-1}}{s^{\alpha}+\lambda_{0}}\left(1+\sum_{k=1}^{\infty}\left(\frac{\sum_{j=1}^{\infty}u^{j}(\lambda_{j-1}-\lambda_{j}) }{s^{\alpha}+\lambda_{0}}\right)^{k}\right),\ \ (\mathrm{using}\ (\ref{fm2}))\nonumber\\
&=\dfrac{s^{\alpha-1}}{s^{\alpha}+\lambda_{0}}\sum_{k=0}^{\infty}\left(\frac{\sum_{j=1}^{\infty}u^{j}(\lambda_{j-1}-\lambda_{j}) }{s^{\alpha}+\lambda_{0}}\right)^{k}\nonumber\\
&=s^{\alpha-1}\left(s^{\alpha}-\sum_{j=0}^{\infty}u^{j}(\lambda_{j-1}-\lambda_{j})\right)^{-1},\label{pgghg1}
\end{align}
which on using (\ref{mi}) gives (\ref{pgfa}).
\end{proof}
Next, we show that the pgf of CFPP solves the following differential equation:
\begin{equation}\label{pgf}
\partial^{\alpha}_{t}G^{\alpha}_{c}(u,t)=G^{\alpha}_{c}(u,t)\sum_{j=0}^{\infty}u^{j}(\lambda_{j-1}-\lambda_{j}),\ \ G^{\alpha}_{c}(u,0)=1. 
\end{equation}
On taking Caputo derivative in $G^{\alpha}_{c}(u,t)=\sum_{n=0}^{\infty}u^{n}p^\alpha_{c}(n,t)$, we get
\begin{align*}
\partial^{\alpha}_{t}G^{\alpha}_{c}(u,t)&=\sum_{n=0}^{\infty}u^{n}\partial^{\alpha}_{t}p^\alpha_{c}(n,t)\\
&=\sum_{n=0}^{\infty}u^{n}\left(-\lambda_{0}p^\alpha_{c}(n,t)+\sum_{j=1}^{n}(\lambda_{j-1}-\lambda_{j})p^\alpha_{c}(n-j,t)\right),\ \ (\mathrm{using}\ (\ref{model}))\\
&=-\lambda_{0}G^{\alpha}_{c}(u,t)+\sum_{n=0}^{\infty}\sum_{j=1}^{n}u^{n}(\lambda_{j-1}-\lambda_{j})p^\alpha_{c}(n-j,t)\\
&=-\lambda_{0}G^{\alpha}_{c}(u,t)+\sum_{j=1}^{\infty}\sum_{n=j}^{\infty}u^{n}(\lambda_{j-1}-\lambda_{j})p^\alpha_{c}(n-j,t)\\
&=-\lambda_{0}G^{\alpha}_{c}(u,t)+\sum_{j=1}^{\infty}\sum_{n=0}^{\infty}u^{n+j}(\lambda_{j-1}-\lambda_{j})p^\alpha_{c}(n,t)\\
&=-\lambda_{0}G^{\alpha}_{c}(u,t)+G^{\alpha}_{c}(u,t)\sum_{j=1}^{\infty}u^{j}(\lambda_{j-1}-\lambda_{j})\\
&=G^{\alpha}_{c}(u,t)\sum_{j=0}^{\infty}u^{j}(\lambda_{j-1}-\lambda_{j}).
\end{align*}
Note that the Laplace transform of the pgf of CFPP can be also obtained from the above result as follows: By taking Laplace transform in (\ref{pgf}), we get
\begin{equation*}
s^{\alpha}\tilde{G}^{\alpha}_{c}(u,s)-s^{\alpha-1}G^{\alpha}_{c}(u,0)=\tilde{G}^{\alpha}_{c}(u,s)\sum_{j=0}^{\infty}u^{j}(\lambda_{j-1}-\lambda_{j}).
\end{equation*}
Thus,
\begin{equation}\label{ltpgf}
\tilde{G}^{\alpha}_{c}(u,s)=s^{\alpha-1}\left(s^{\alpha}-\sum_{j=0}^{\infty}u^{j}(\lambda_{j-1}-\lambda_{j})\right)^{-1},
\end{equation}
which coincides with (\ref{pgghg1}).
\begin{remark}
If the difference of intensities is a constant, {\it i.e.}, $\lambda_{j-1}-\lambda_{j}=\delta$ for all $j\ge1$ then (\ref{pgf}) reduces to
\begin{equation*}
\partial^{\alpha}_{t}G^{\alpha}_{c}(u,t)=G^{\alpha}_{c}(u,t)\left(\frac{\delta u}{1-u}-\lambda_{0}\right).
\end{equation*}
\end{remark}
Next, we obtain the mean and variance of CFPP using its pgf. From (\ref{pgfa}), we have 
\begin{equation}\label{smpgf}
G^{\alpha}_{c}(u,t)=\sum_{k=0}^{\infty}\frac{t^{k\alpha}}{\Gamma(k\alpha+1)}\left(\sum_{j=0}^{\infty}(\lambda_{j-1}-\lambda_{j}) u^{j}\right)^{k}.
\end{equation}	
On taking the derivatives, we get
\begin{equation*}
\frac{\partial G^{\alpha}_{c}(u,t)}{\partial u} =\sum_{k=1}^{\infty}\frac{kt^{k\alpha}}{\Gamma(k\alpha+1)}\left(\sum_{j=0}^{\infty}(\lambda_{j-1}-\lambda_{j}) u^{j}\right)^{k-1}\left(\sum_{j=1}^{\infty}(\lambda_{j-1}-\lambda_{j})j u^{j-1}\right).
\end{equation*}
and
\begin{align*}
\frac{\partial^{2}G^{\alpha}_{c}(u,t)}{\partial u^{2}} &=\sum_{k=2}^{\infty}\frac{k(k-1)t^{k\alpha}}{\Gamma(k\alpha+1)}\left(\sum_{j=0}^{\infty}(\lambda_{j-1}-\lambda_{j}) u^{j}\right)^{^{k-2}}\left(\sum_{j=1}^{\infty}(\lambda_{j-1}-\lambda_{j})j u^{j-1}\right)^{2}\\
&\ \ \ \ +\sum_{k=1}^{\infty}\frac{kt^{k\alpha}}{\Gamma(k\alpha+1)}\left(\sum_{j=0}^{\infty}(\lambda_{j-1}-\lambda_{j}) u^{j}\right)^{k-1}\left(\sum_{j=2}^{\infty}(\lambda_{j-1}-\lambda_{j})j(j-1)u^{j-2}\right).
\end{align*}
Now, the mean of CFPP is given by
\begin{align}\label{wswee11}
\mathbb{E}\left(\mathcal{N}^{\alpha}_{c}(t)\right)&=\frac{\partial G^{\alpha}_{c}(u,t)}{\partial u}\bigg|_{u=1}\nonumber\\
&=\sum_{k=1}^{\infty}\frac{kt^{k\alpha}}{\Gamma(k\alpha+1)}\left(\sum_{j=0}^{\infty}(\lambda_{j-1}-\lambda_{j})\right)^{k-1}\left(\sum_{j=1}^{\infty}(\lambda_{j-1}-\lambda_{j})j\right)\nonumber\\
&=\frac{t^{\alpha}}{\Gamma(\alpha+1)}\sum_{j=0}^{\infty}\lambda_{j},
\end{align}
using (\ref{asdesa11}) in the last step. Also, its variance can be obtained as follows:
\begin{align*}
\mathbb{E}\left(\mathcal{N}^{\alpha}_{c}(t)(\mathcal{N}^{\alpha}_{c}(t)-1)\right)&=\frac{\partial^{2} G^{\alpha}_{c}(u,t)}{\partial u^{2}}\bigg|_{u=1}\\
&=\sum_{k=2}^{\infty}\frac{k(k-1)t^{k\alpha}}{\Gamma(k\alpha+1)}\left(\sum_{j=0}^{\infty}(\lambda_{j-1}-\lambda_{j}) \right)^{^{k-2}}\left(\sum_{j=1}^{\infty}(\lambda_{j-1}-\lambda_{j})j \right)^{2}\\
&\ \ \ \ +\sum_{k=1}^{\infty}\frac{kt^{k\alpha}}{\Gamma(k\alpha+1)}\left(\sum_{j=0}^{\infty}(\lambda_{j-1}-\lambda_{j}) \right)^{k-1}\left(\sum_{j=2}^{\infty}(\lambda_{j-1}-\lambda_{j})j(j-1) \right)\\
&=\frac{2t^{2\alpha}}{\Gamma(2\alpha+1)}\left(\sum_{j=0}^{\infty}\lambda_{j}\right)^{2}+\frac{2t^{\alpha}}{\Gamma(\alpha+1)}\sum_{j=1}^{\infty}j\lambda_{j}.
\end{align*}	
Thus, 
\begin{equation*}
\mathbb{E}\left(\mathcal{N}^{\alpha}_{c}(t)^2\right)=\frac{2t^{2\alpha}}{\Gamma(2\alpha+1)}\left(\sum_{j=0}^{\infty}\lambda_{j}\right)^{2}+\frac{2t^{\alpha}}{\Gamma(\alpha+1)}\sum_{j=1}^{\infty}j\lambda_{j}+\frac{t^{\alpha}}{\Gamma(\alpha+1)}\sum_{j=0}^{\infty}\lambda_{j}.
\end{equation*}
Hence, the variance of CFPP is given by
\begin{equation}\label{var}
\operatorname{ Var}\left(\mathcal{N}^{\alpha}_{c}(t)\right)=\frac{t^{\alpha}\sum_{j=0}^{\infty}\lambda_{j}}{\Gamma(\alpha+1)}+\frac{2t^{\alpha}\sum_{j=1}^{\infty}j\lambda_{j}}{\Gamma(\alpha+1)}+\frac{2\left(t^{\alpha}\sum_{j=0}^{\infty}\lambda_{j}\right)^{2}}{\Gamma(2\alpha+1)}-\frac{\left(t^{\alpha}\sum_{j=0}^{\infty}\lambda_{j}\right)^{2}}{\Gamma^2(\alpha+1)}.
\end{equation}
The pgf of CFPP can also be utilized to obtain its factorial moments as follows:
\begin{proposition}
The $r$th factorial moment of the CFPP $\psi^\alpha_c(r,t)=	\mathbb{E}(\mathcal{N}^{\alpha}_{c}(t)(\mathcal{N}^{\alpha}_{c}(t)-1)\dots(\mathcal{N}^{\alpha}_{c}(t)-r+1))$, $r\ge1$, is given by
		\begin{equation*}
			\psi^\alpha_c(r,t)=r!\sum_{k=1}^{r}\frac{t^{k\alpha}}{\Gamma(k\alpha+1)}\underset{m_j\in\mathbb{N}_0}{\underset{\sum_{j=1}^km_j=r}{\sum}}\prod_{\ell=1}^{k}\left(\frac{1}{m_\ell!}\sum_{j=0}^{\infty}(j)_{m_\ell}(\lambda_{j-1}-\lambda_{j})\right),
		\end{equation*}
		where $(j)_{m_\ell}=j(j-1)\dots(j-m_\ell+1)$ denotes the falling factorial.
	\end{proposition}
	\begin{proof}
		From (\ref{pgfa}), we get
		\begin{align}\label{ttt}
			\psi^\alpha_c(r,t)&=\frac{\partial^{r}G^{\alpha}_{c}(u,t)}{\partial u^{r}}\bigg|_{u=1}\nonumber\\
			&=\sum_{k=0}^{r}\frac{1}{k!}E^{(k)}_{\alpha,1}\left(t^{\alpha}\sum_{j=0}^{\infty}(\lambda_{j-1}-\lambda_{j})u^{j}\right)	\left.A_{r,k}\left(t^{\alpha}\sum_{j=0}^{\infty}(\lambda_{j-1}-\lambda_{j})u^{j}\right)\right|_{u=1},
		\end{align}
		where we have used the $r$th derivative of composition of two functions (see Johnson (2002), Eq. (3.3)). Here,
		\begin{align}\label{mkgtrr4543t}
			A_{r,k}&\left(t^{\alpha}\sum_{j=0}^{\infty}(\lambda_{j-1}-\lambda_{j})u^{j}\right)\Bigg|_{u=1}\nonumber\\
			&=\sum_{m=0}^{k}\frac{k!}{m!(k-m)!}\left(-t^{\alpha}\sum_{j=0}^{\infty}(\lambda_{j-1}-\lambda_{j})u^{j}\right)^{k-m}\frac{\mathrm{d}^{r}}{\mathrm{d}u
				^{^{r}}}\left(t^{\alpha}\sum_{j=0}^{\infty}(\lambda_{j-1}-\lambda_{j})u^{j}\right)^{m}\Bigg|_{u=1}\nonumber\\
			&=t^{k\alpha }\frac{\mathrm{d}^{r}}{\mathrm{d}u^{^{r}}}\left(\sum_{j=0}^{\infty}(\lambda_{j-1}-\lambda_{j})u^{j}\right)^{k}\Bigg|_{u=1},
		\end{align}
		where the last step follows by using (\ref{asdesa11}). From (\ref{re}), we get
		\begin{align}\label{ppp}
			E^{(k)}_{\alpha,1}\left(t^{\alpha}\sum_{j=0}^{\infty}(\lambda_{j-1}-\lambda_{j})u^{j}\right)\Bigg|_{u=1}&=k!E^{k+1}_{\alpha,k\alpha+1}\left(t^{\alpha}\sum_{j=0}^{\infty}(\lambda_{j-1}-\lambda_{j})u^{j}\right)\Bigg|_{u=1}\nonumber\\
			&=\frac{k!}{\Gamma(k\alpha+1)}.
		\end{align}
		Now, by using the following result (see Johnson (2002), Eq. (3.6))
		\begin{equation}\label{qlkju76}
		\frac{\mathrm{d}^{r}}{\mathrm{d}w^{^{r}}}(f(w))^{k}=\underset{m_j\in\mathbb{N}_0}{\underset{m_{1}+m_{2}+\dots+m_{k}=r}{\sum}}\frac{r!}{m_1!m_2!\ldots m_k!}f^{(m_{1})}(w)f^{(m_{2})}(w)\dots f^{(m_{k})}(w),
		\end{equation}
		we get		
		\begin{align}\label{ccc}
			\frac{\mathrm{d}^{r}}{\mathrm{d}u^{^{r}}}\left(\sum_{j=0}^{\infty}(\lambda_{j-1}-\lambda_{j})u^{j}\right)^{k}\Bigg|_{u=1}&=r!\underset{m_j\in\mathbb{N}_0}{\underset{\sum_{j=1}^km_j=r}{\sum}}\prod_{\ell=1}^{k}\frac{1}{m_\ell!}\frac{\mathrm{d}^{m_{\ell}}}{\mathrm{d}u^{{m_{\ell}}}}\left(\sum_{j=0}^{\infty}(\lambda_{j-1}-\lambda_{j})u^{j}\right)\Bigg|_{u=1}\nonumber\\
			&=r!\underset{m_j\in\mathbb{N}_0}{\underset{\sum_{j=1}^km_j=r}{\sum}}\prod_{\ell=1}^{k}\frac{1}{m_\ell!}\sum_{j=0}^{\infty}(j)_{m_\ell}(\lambda_{j-1}-\lambda_{j}).
		\end{align}	
		Note that the expression in right hand side of (\ref{ccc}) vanishes for $k=0$. Finally, on substituting (\ref{mkgtrr4543t}), (\ref{ppp}) and (\ref{ccc}) in (\ref{ttt}), we get the required result.	
\end{proof}
Next, we obtain the moment generating function (mgf) of CFPP on non-positive support.
\begin{proposition}
	The mgf $m^\alpha_c(w,t)=\mathbb{E}(e^{-w\mathcal{N}^{\alpha}_{c}(t)})$, $w\geq0$, of CFPP is given by
	\begin{equation}\label{mjhgss}
	m^\alpha_c(w,t)=E_{\alpha,1}\left(t^{\alpha}\sum_{j=0}^{\infty}(\lambda_{j-1}-\lambda_{j})e^{-wj}\right).
	\end{equation}
\end{proposition}
\begin{proof} Using (\ref{dssdew1}), we have
	\begin{align*}
	m^\alpha_c(w,t)&=p^\alpha_{c}(0,t)+\sum_{n=1}^{\infty}e^{-wn}p^\alpha_{c}(n,t)\\
	&=E_{\alpha,1}(-\lambda_{0}t^{\alpha})+\sum_{n=1}^{\infty}e^{-wn}\sum_{k=1}^{n}\sum_{\Lambda_{n}^{k}}k!\prod_{j=1}^{n-k+1}\frac{(\lambda_{j-1}-\lambda_{j})^{k_{j}}}{k_{j}!} t^{k\alpha}E_{\alpha,k\alpha+1}^{k+1}(-\lambda_{0}t^{\alpha})\\
	&=E_{\alpha,1}(-\lambda_{0}t^{\alpha})+\sum_{k=1}^{\infty}t^{k\alpha}E_{\alpha,k\alpha+1}^{k+1}(-\lambda_{0}t^{\alpha})\sum_{n=k}^{\infty}\sum_{\Lambda_{n}^{k}}k!\prod_{j=1}^{n-k+1}\frac{(\lambda_{j-1}-\lambda_{j})^{k_{j}}}{k_{j}!}e^{-wn}\\
	&=E_{\alpha,1}(-\lambda_{0}t^{\alpha})+\sum_{k=1}^{\infty}\left(\sum_{j=1}^{\infty}(\lambda_{j-1}-\lambda_{j})e^{-wj}\right)^{k} t^{k\alpha}E_{\alpha,k\alpha+1}^{k+1}(-\lambda_{0}t^{\alpha}),\ \ \text{(using\ (\ref{fm2}))}\\
	&=\sum_{k=0}^{\infty}\left(t^{\alpha}\sum_{j=1}^{\infty}(\lambda_{j-1}-\lambda_{j})e^{-wj}\right)^{k}E_{\alpha,k\alpha+1}^{k+1}(-\lambda_{0}t^{\alpha})\\
	&=E_{\alpha,1}\left(t^{\alpha}\sum_{j=0}^{\infty}(\lambda_{j-1}-\lambda_{j})e^{-wj}\right),
	\end{align*}
	where we have used (\ref{formula}) in the last step.
\end{proof}
The mgf of CFPP solves the following fractional differential equation:
	\begin{equation*}
	\partial^{\alpha}_{t}m^{\alpha}_c(w,t)=m^{\alpha}_c(w,t)\sum_{j=0}^{\infty}e^{-wj}(\lambda_{j-1}-\lambda_{j}),\ \ m^{\alpha}_c(w,0)=1. 
	\end{equation*}
The above equation can be solved by using the Laplace transform method to obtain the mgf (\ref{mjhgss}). The proof follows similar lines to that of the related result for the pgf of CFPP.
	
The mean and variance of the CFPP can also be obtained from its mgf. From (\ref{mjhgss}), we have
\begin{equation*}
	m^\alpha_c(w,t)=\sum_{k=0}^{\infty}\frac{t^{k\alpha}}{\Gamma(k\alpha+1)}\left(\sum_{j=0}^{\infty}(\lambda_{j-1}-\lambda_{j}) e^{-wj}\right)^{k}.
\end{equation*}
On taking the derivatives, we get
\begin{equation*}
	\frac{\partial m^\alpha_c(w,t)}{\partial w} =\sum_{k=1}^{\infty}\frac{kt^{k\alpha}}{\Gamma(k\alpha+1)}\left(\sum_{j=0}^{\infty}(\lambda_{j-1}-\lambda_{j}) e^{-wj}\right)^{k-1}\left(\sum_{j=1}^{\infty}(\lambda_{j-1}-\lambda_{j})(-j)e^{-wj}\right),
\end{equation*}
and 
\begin{align*}
	\frac{\partial^{2}m^\alpha_c(w,t)}{\partial w^{2}} &=\sum_{k=2}^{\infty}\frac{k(k-1)t^{k\alpha}}{\Gamma(k\alpha+1)}\left(\sum_{j=0}^{\infty}(\lambda_{j-1}-\lambda_{j}) e^{-wj}\right)^{^{k-2}}\left(\sum_{j=1}^{\infty}(\lambda_{j-1}-\lambda_{j})j e^{-wj}\right)^{2}\\
	&\ \ \ \ +\sum_{k=1}^{\infty}\frac{kt^{k\alpha}}{\Gamma(k\alpha+1)}\left(\sum_{j=0}^{\infty}(\lambda_{j-1}-\lambda_{j}) e^{-wj}\right)^{k-1}\left(\sum_{j=1}^{\infty}(\lambda_{j-1}-\lambda_{j})j^{2} e^{-wj}\right).
\end{align*}
Now, the mean of CFPP is given by
\begin{align*}
	\mathbb{E}\left(\mathcal{N}^{\alpha}_{c}(t)\right)&=-\frac{\partial m^\alpha_c(w,t)}{\partial w}\bigg|_{w=0}\\
	&=\sum_{k=1}^{\infty}\frac{kt^{k\alpha}}{\Gamma(k\alpha+1)}\left(\sum_{j=0}^{\infty}(\lambda_{j-1}-\lambda_{j})\right)^{k-1}\left(\sum_{j=1}^{\infty}(\lambda_{j-1}-\lambda_{j})j\right)\\
	&=\frac{t^{\alpha}}{\Gamma(\alpha+1)}\sum_{j=0}^{\infty}\lambda_{j},
\end{align*}
which agrees with (\ref{wswee11}). Its second order moment is given by
\begin{align*}
	\mathbb{E}\left((\mathcal{N}^{\alpha}_{c}(t))^2\right)&=\frac{\partial^{2}m^\alpha_c(w,t)}{\partial w^{2}}\bigg|_{w=0}\\
	&=\sum_{k=2}^{\infty}\frac{k(k-1)t^{k\alpha}}{\Gamma(k\alpha+1)}\left(\sum_{j=0}^{\infty}(\lambda_{j-1}-\lambda_{j})\right)^{^{k-2}}\left(\sum_{j=1}^{\infty}(\lambda_{j-1}-\lambda_{j})j \right)^{2}\\
	&\ \ \ \ +\sum_{k=1}^{\infty}\frac{kt^{k\alpha}}{\Gamma(k\alpha+1)}\left(\sum_{j=0}^{\infty}(\lambda_{j-1}-\lambda_{j})\right)^{k-1}\left(\sum_{j=1}^{\infty}(\lambda_{j-1}-\lambda_{j})(j(j-1)+j)\right)\\
	&=\frac{2t^{2\alpha}}{\Gamma(2\alpha+1)}\left(\sum_{j=0}^{\infty}\lambda_{j}\right)^{2}+\frac{t^{\alpha}}{\Gamma(\alpha+1)}\left(\sum_{j=0}^{\infty}\lambda_{j}+2\sum_{j=1}^{\infty}j\lambda_{j}\right).
\end{align*}
The variance (\ref{var}) can now be obtained by computing $\mathbb{E}\left((\mathcal{N}^{\alpha}_{c}(t))^2\right)-\left(\mathbb{E}\left(\mathcal{N}^{\alpha}_{c}(t)\right)\right)^{2}$.
\begin{proposition}
The $r$th moment $\mu^\alpha_c(r,t)=\mathbb{E}\left((\mathcal{N}^{\alpha}_{c}(t))^r\right)$, $r\ge1$, of CFPP is given by
\begin{equation*}
\mu^\alpha_c(r,t)=r!\sum_{k=1}^{r}\frac{t^{k\alpha}}{\Gamma(k\alpha+1)}\underset{m_j\in\mathbb{N}_0}{\underset{\sum_{j=1}^km_j=r}{\sum}}\prod_{\ell=1}^{k}\left(\frac{1}{m_\ell!}\sum_{j=0}^{\infty}j^{m_\ell}(\lambda_{j-1}-\lambda_{j})\right).
\end{equation*}
\end{proposition}
\begin{proof}
Taking the $r$th derivative of composition of two functions (see Johnson (2002), Eq. (3.3)), we get
\begin{align}\label{tt}
\mu^\alpha_c(r,t)&=(-1)^{r}\frac{\partial^{r} m^\alpha_c(w,t)}{\partial w^{r}}\bigg|_{w=0}\nonumber\\
&=\sum_{k=0}^{r}\frac{(-1)^{r}}{k!}E^{(k)}_{\alpha,1}\left(t^{\alpha}\sum_{j=0}^{\infty}(\lambda_{j-1}-\lambda_{j})e^{-wj}\right)	\left.B_{r,k}\left(t^{\alpha}\sum_{j=0}^{\infty}(\lambda_{j-1}-\lambda_{j})e^{-wj}\right)\right|_{w=0},
\end{align}
where 
\begin{align}\label{mkgtrr4543}
B_{r,k}&\left(t^{\alpha}\sum_{j=0}^{\infty}(\lambda_{j-1}-\lambda_{j})e^{-wj}\right)\Bigg|_{w=0}\nonumber\\
&=\sum_{m=0}^{k}\frac{k!}{m!(k-m)!}\left(-t^{\alpha}\sum_{j=0}^{\infty}(\lambda_{j-1}-\lambda_{j})e^{-wj}\right)^{k-m}\frac{\mathrm{d}^{r}}{\mathrm{d}w^{^{r}}}\left(t^{\alpha}\sum_{j=0}^{\infty}(\lambda_{j-1}-\lambda_{j})e^{-wj}\right)^{m}\Bigg|_{w=0}\nonumber\\
&=t^{k\alpha }\frac{\mathrm{d}^{r}}{\mathrm{d}w^{^{r}}}\left(\sum_{j=0}^{\infty}(\lambda_{j-1}-\lambda_{j})e^{-wj}\right)^{k}\Bigg|_{w=0}.
\end{align}
The last equality follows by using (\ref{asdesa11}). Now, from (\ref{re}), we get
		\begin{align}\label{pp}
			E^{(k)}_{\alpha,1}\left(t^{\alpha}\sum_{j=0}^{\infty}(\lambda_{j-1}-\lambda_{j})e^{-wj}\right)\Bigg|_{w=0}&=k!E^{k+1}_{\alpha,k\alpha+1}\left(t^{\alpha}\sum_{j=0}^{\infty}(\lambda_{j-1}-\lambda_{j})e^{-wj}\right)\Bigg|_{w=0}\nonumber\\
			&=\frac{k!}{\Gamma(k\alpha+1)}.
		\end{align}		
Using (\ref{qlkju76}), we have
		\begin{align}\label{cc}
			\frac{\mathrm{d}^{r}}{\mathrm{d}w^{^{r}}}\left(\sum_{j=0}^{\infty}(\lambda_{j-1}-\lambda_{j})e^{-wj}\right)^{k}\Bigg|_{w=0}&=r!\underset{m_j\in\mathbb{N}_0}{\underset{\sum_{j=1}^km_j=r}{\sum}}\prod_{\ell=1}^{k}\frac{1}{m_\ell!}\frac{\mathrm{d}^{m_{\ell}}}{\mathrm{d}w^{{m_{\ell}}}}\left(\sum_{j=0}^{\infty}(\lambda_{j-1}-\lambda_{j})e^{-wj}\right)\Bigg|_{w=0}\nonumber\\
			&=(-1)^rr!\underset{m_j\in\mathbb{N}_0}{\underset{\sum_{j=1}^km_j=r}{\sum}}\prod_{\ell=1}^{k}\frac{1}{m_\ell!}\sum_{j=0}^{\infty}j^{m_\ell}(\lambda_{j-1}-\lambda_{j}).
		\end{align}	
The result follows on substituting (\ref{mkgtrr4543})-(\ref{cc}) in (\ref{tt}).	
\end{proof}
In the following result it is shown that the CFPP is equal in distribution to a compound fractional Poisson process. Thus, it is neither Markovian nor a L\'evy process.
\begin{theorem}
	Let $\{N^\alpha(t)\}_{t\ge0}$, $0<\alpha\le1$, be the TFPP with intensity parameter $\lambda_{0}>0$ and $\{X_{i}\}_{i\ge1}$ be a sequence of independent and identically distributed (iid) random variables with the following distribution:
	\begin{equation}\label{ngftt4}
	\mathrm{Pr}\{X_{i}=j\}=\frac{\lambda_{j-1}-\lambda_{j}}{\lambda_{0}}, \ \ j\ge1.
	\end{equation}
	Then,
	\begin{equation}\label{cd}
	\mathcal{N}^{\alpha}_{c}(t)\overset{d}{=}\sum_{i=1}^{N^\alpha(t)}X_{i}, \ \ \ t\ge0,
	\end{equation}
	where $\{X_{i}\}_{i\ge1}$ is independent of $\{N^\alpha(t)\}_{t\ge0}$.
\end{theorem}
\begin{proof}
	The mgf of $N^\alpha(t)$, $t\ge0$, is given by (see Laskin (2003), Eq. (35))
	\begin{equation}\label{mgft}
	\mathbb{E}\left(e^{-wN^\alpha(t)}\right)=E_{\alpha,1}(\lambda_{0}t^{\alpha}(e^{-w}-1)),\ \ w\ge0.
	\end{equation}
	Also, the mgf of $X_{i}$, $i\ge1$, can be obtained as
	\begin{equation}\label{plo876}
	\mathbb{E}\left(e^{-wX_{i}}\right)=\frac{1}{\lambda_{0}}\sum_{j=1}^{\infty}(\lambda_{j-1}-\lambda_{j})e^{-jw}.
	\end{equation}
	Now,
	\begin{align*}
	\mathbb{E}\left(e^{-w\sum_{i=1}^{N^\alpha(t)}X_{i}}\right)&=\mathbb{E}\left(\mathbb{E}\left(e^{-w\sum_{i=1}^{N^\alpha(t)}X_{i}}\big|N^\alpha(t)\right)\right)\\
	&=\mathbb{E}\left(\left(\mathbb{E}\left(e^{-wX_{1}}\right)\right)^{N^\alpha(t)}\right)\\
	&=\mathbb{E}\left(\exp\left(N^\alpha(t)\ln\left(\frac{1}{\lambda_{0}}\sum_{j=1}^{\infty}(\lambda_{j-1}-\lambda_{j})e^{-jw}\right)\right)\right),\ \ (\text{using (\ref{plo876})})\\
	&=E_{\alpha,1}\left(\lambda_{0}t^{\alpha}\left(\exp\left(\ln\frac{1}{\lambda_{0}}\sum_{j=1}^{\infty}(\lambda_{j-1}-\lambda_{j})e^{-wj}\right)-1\right)\right),\ \ (\text{using (\ref{mgft})})\\
	&=E_{\alpha,1}\left(t^{\alpha}\sum_{j=0}^{\infty}(\lambda_{j-1}-\lambda_{j})e^{-wj}\right),
	\end{align*}
	which agrees with the mgf of CFPP (\ref{mjhgss}). This completes the proof.
\end{proof}
\begin{remark}
	The one-dimensional distribution of TFPP is given by
	\begin{equation*}
	\mathrm{Pr}\{N^\alpha(t)=n\}=(\lambda_{0}t^{\alpha})^{n}E_{\alpha,n\alpha+1}^{n+1}(-\lambda_{0}t^{\alpha}),\ \ n\ge0.
	\end{equation*}
	For $n=0$, using (\ref{cd}) we have
	\begin{equation*}
	\mathrm{Pr}\{\mathcal{N}^\alpha_c(t)=0\}=\mathrm{Pr}\{N^\alpha(t)=0\}=E_{\alpha,1}(-\lambda_{0}t^{\alpha}).
	\end{equation*}
	As $X_i$'s are independent of $N^\alpha(t)$, for $n\ge1$, we get
	\begin{align}\label{rdee32}
	\mathrm{Pr}\{\mathcal{N}^\alpha_c(t)=n\}&=\sum_{k=1}^{n}\mathrm{Pr}\{X_{1}+X_{2}+\dots+X_{k}=n\}\mathrm{Pr}\{N^\alpha(t)=k\}\\
	&=\sum_{k=1}^{n}\sum_{\Theta_{n}^{k}}k!\prod_{j=1}^{n}\frac{(\lambda_{j-1}-\lambda_{j})^{k_{j}}}{k_{j}!}t^{k\alpha}E_{\alpha,k\alpha+1}^{k+1}(-\lambda_{0}t^{\alpha}),\nonumber	
	\end{align}
where $k_j$ is the total number of claims of $j$ units. The above expression agrees with (\ref{dist}).

As $X_{i}$'s are iid, we have
	\begin{align}\label{sweee32}
	\mathrm{Pr}\{X_{1}+X_{2}+\dots+X_{k}=n\}&=\underset{m_j\in\mathbb{N}}{\underset{m_{1}+m_{2}+\dots+m_{k}=n}{\sum}}\mathrm{Pr}\{X_{1}=m_1,X_{2}=m_2,\ldots,X_{k}=m_k\}\nonumber\\
	&=\underset{m_j\in\mathbb{N}}{\underset{m_{1}+m_{2}+\dots+m_{k}=n}{\sum}}\prod_{j=1}^{k}\mathrm{Pr}\{X_{j}=m_j\}\nonumber\\
	&=\underset{m_j\in\mathbb{N}}{\underset{m_{1}+m_{2}+\dots+m_{k}=n}{\sum}}\frac{1}{\lambda_{0}^{k}}\prod_{j=1}^{k}(\lambda_{m_j-1}-\lambda_{m_j}),	
	\end{align}
where we have used (\ref{ngftt4}). Substituting (\ref{sweee32}) in (\ref{rdee32}), we get an equivalent expression for the one-dimensional distribution of the CFPP as
\begin{equation*}
\mathrm{Pr}\{\mathcal{N}^\alpha_c(t)=n\}=\sum_{k=1}^{n}\underset{m_j\in\mathbb{N}}{\underset{m_{1}+m_{2}+\dots+m_{k}=n}{\sum}}\prod_{j=1}^{k}(\lambda_{m_j-1}-\lambda_{m_j})t^{k\alpha}E_{\alpha,k\alpha+1}^{k+1}(-\lambda_{0}t^{\alpha}).	
\end{equation*}
\end{remark}

\section{Convoluted Poisson process: A Special case of CFPP}\label{Section4}
Here, we discuss a particular case of the CFPP, namely, the convoluted Poisson process (CPP) which we denote by $\{\mathcal{N}_{c}(t)\}_{t\ge 0}$.

The CPP is defined as a stochastic process whose state probabilities satisfy 
\begin{equation*}
	\frac{\mathrm{d}}{\mathrm{d}t}p_{c}(n,t)=-\lambda_{n}*p_{c}(n,t)+\lambda_{n-1}*p_{c}(n-1,t),
\end{equation*}
with $p_{c}(n,0)=\delta_{0}(n)$, $n\ge0$. The conditions on $\lambda_{n}$'s are same as in the case of CFPP.

For $\alpha=1$, the CFPP reduces to CPP. Thus, its the state probabilities are given by
\begin{equation}\label{key1q112}
	p_{c}(n,t)=\begin{cases}
		e^{-\lambda_{0}t},\ \ n=0,\vspace*{.2cm}\\
		\displaystyle\sum_{k=1}^{n}\sum_{\Lambda_{n}^{k}}\prod_{j=1}^{n-k+1}\frac{(\lambda_{j-1}-\lambda_{j})^{k_{j}}}{k_{j}!} t^{k}e^{-\lambda_{0}t},\ \ n\ge1.
	\end{cases}
\end{equation}
On substituting $\lambda_{0}=\lambda$ and $\lambda_{n}=0$ for $n\ge1$ in (\ref{key1q112}), we get
\begin{equation*}
p(n,t)=\frac{(\lambda t)^ne^{-\lambda t}}{n!},\ \  n\ge0,
\end{equation*}
which is the distribution of Poisson process. Thus, Poisson process is a particular case of the CPP.
\begin{remark}
	Let $W_{c}$ be the first waiting time of CPP. Then,
	\begin{equation*}
	\mathrm{Pr}\{W_{c}>t\}=\mathrm{Pr}\{\mathcal{N}_{c}(t)=0\}=e^{-\lambda_{0}t}, \ t>0,
	\end{equation*}
	which coincides with the first waiting time of Poisson process. It is known that the Poisson process is a renewal process. Thus, it implies that the CPP is not a renewal process as the one-dimensional distributions of Poisson process and CPP are different.	
\end{remark}
A direct method to obtain the pgf of CPP is as follows: 
\begin{align}\label{kjhgt21}
G_c(u,t)&=p_{c}(0,t)+\sum_{n=1}^{\infty}u^{n}p_{c}(n,t)\nonumber\\
&=e^{-t\lambda_{0}} +\sum_{n=1}^{\infty}u^{n}\sum_{k=1}^{n}\sum_{\Lambda_{n}^{k}}\prod_{j=1}^{n-k+1}\frac{(\lambda_{j-1}-\lambda_{j})^{k_{j}}}{k_{j}!} t^{k}e^{-t\lambda_{0}}\\
&=e^{-t\lambda_{0}} \left(1+\sum_{n=1}^{\infty}\sum_{k=1}^{n}\sum_{\Lambda_{n}^{k}}\prod_{j=1}^{n-k+1}\frac{(\lambda_{j-1}-\lambda_{j})^{k_{j}}}{k_{j}!}u^{n}t^{k} \right)\nonumber\\
&=e^{-t\lambda_{0}}\exp\left(t\sum_{j=1}^{\infty}u^{j}(\lambda_{j-1}-\lambda_{j})\right),\ \ \text{(using\ (\ref{fm1}))}\nonumber\\
&=\exp\left(t\sum_{j=0}^{\infty}u^{j}(\lambda_{j-1}-\lambda_{j})\right).\nonumber
\end{align}
Also, its mean and variance are given by
\begin{equation}\label{4.3ew}
	\mathbb{E}\left(\mathcal{N}_{c}(t)\right)=t\sum_{j=0}^{\infty}\lambda_{j}
\end{equation}
and
\begin{equation}\label{4.4er}
	\operatorname{ Var}\left(\mathcal{N}_{c}(t)\right)=t\left(\sum_{j=0}^{\infty}\lambda_{j}+2\sum_{j=1}^{\infty}j\lambda_{j}\right),
\end{equation} 
respectively. 
\begin{remark}
The CPP exhibits overdispersion as $\operatorname{ Var}\left(\mathcal{N}_{c}(t)\right)-\mathbb{E}\left(\mathcal{N}_{c}(t)\right)=2t\sum_{j=1}^{\infty}j\lambda_{j}>0$ for $t>0$ provided $\lambda_{n}\neq0$ for all $n\ge1$.
\end{remark}
\begin{theorem}\label{thwe22}
	The CPP is a L\'evy process.
\end{theorem}
\begin{proof}
	The characteristic function $\phi_c(\xi,t)$ of CPP is given by
	\begin{align*}
	\phi_c(\xi,t)=\mathbb{E}\left(e^{i\xi\mathcal{N}_{c}(t)}\right)&=\sum_{n=0}^{\infty}e^{i\xi n}p_{c}(n,t),\ \ \xi\in\mathbb{R}\\
	&=e^{-t\lambda_{0}}\left(1 +\sum_{n=1}^{\infty}e^{i\xi n}\sum_{k=1}^{n}\sum_{\Lambda_{n}^{k}}\prod_{j=1}^{n-k+1}\frac{(\lambda_{j-1}-\lambda_{j})^{k_{j}}}{k_{j}!} t^{k}\right)\\
	&=e^{-t\lambda_{0}}\exp\left(t\sum_{j=1}^{\infty}e^{i\xi j}(\lambda_{j-1}-\lambda_{j})\right),\ \ \text{(using\ (\ref{fm1}))}\\
	&=\exp\left(-t\sum_{j=0}^{\infty}e^{i\xi j}(\lambda_{j}-\lambda_{j-1})\right).		
	\end{align*}
So, its characteristic exponent is  
	\begin{equation*}
	\psi(\xi)=\sum_{j=0}^{\infty}(\lambda_{j}-\lambda_{j-1})e^{i\xi j}
	=\sum_{j=0}^{\infty}(\lambda_{j-1}-\lambda_{j})(1-e^{i\xi j}).
	\end{equation*}
Thus, the CPP is a l\'evy process with l\'evy measure $\Pi(\mathrm{d}x)=\sum_{j=0}^{\infty}(\lambda_{j-1}-\lambda_{j})\delta_{j}$ where $\delta_{j}$'s are Dirac measures.
\end{proof}
\begin{remark}
Substituting $\alpha=1$ in (\ref{cd}), we get 
\begin{equation*}
\mathcal{N}_{c}(t)\overset{d}{=}\sum_{i=1}^{N(t)}X_{i}, \ \ \ t\ge0,
\end{equation*}
where $\{X_{i}\}_{i\ge1}$ is independent of the Poisson process $\{N(t)\}_{t\ge0}$, {\it i.e.}, the CPP is a compound Poisson process. Hence, it's a L\'evy process.  
\end{remark}
Let $\{T_{2\alpha}(t)\}_{t>0}$ be a random process whose distribution is given by the folded solution of following fractional diffusion equation (see Orsingher and Beghin (2004)): 
\begin{equation}\label{diff}
\partial^{2\alpha}_{t}u(x,t)=\frac{\partial^{2}}{\partial x^{2}}u(x,t),\ \ x\in\mathbb{R},\ t>0,
\end{equation}
with $u(x,0)=\delta(x)$ for $0<\alpha\le 1$ and $\frac{\partial^{}}{\partial t}u(x,0)=0$ for $1/2<\alpha\le1$.

The Laplace transform of the folded solution $f_{T_{2\alpha}}(x,t)$ of (\ref{diff}) is given by (see Orsingher and Polito (2010), Eq. (2.29))
\begin{equation}\label{lta}
\int_{0}^{\infty}e^{-st}f_{T_{2\alpha}}(x,t)\mathrm{d}t=s^{\alpha-1}e^{-x s^{\alpha}}, \ \ x>0.
\end{equation}	
The next result establish a time-changed relationship between the CPP and CFPP.
\begin{theorem}
	Let the process $\{T_{2\alpha}(t)\}_{t>0}$, $0<\alpha\le 1$, be independent of the CPP $\{\mathcal{N}_{c}(t)\}_{t>0}$. Then, 
	\begin{equation}\label{sub}
		\mathcal{N}^{\alpha}_{c}(t)\overset{d}{=}\mathcal{N}_{c}(T_{2\alpha}(t)).
	\end{equation}
\end{theorem}
\begin{proof}
From (\ref{mjhgqq1}), we have
	\begin{align}\label{sin}
\tilde{G}^{\alpha}_{c}(u,s)&=\dfrac{s^{\alpha-1}}{s^{\alpha}+\lambda_{0}}+\sum_{n=1}^{\infty}u^{n}\sum_{k=1}^{n}\sum_{\Lambda_{n}^{k}}k!\prod_{j=1}^{n-k+1}\frac{(\lambda_{j-1}-\lambda_{j})^{k_{j}}}{k_{j}!}\frac{s^{\alpha-1}}{(s^{\alpha}+\lambda_{0})^{k+1}},\ \ s>0\nonumber\\
		&=\int_{0}^{\infty}s^{\alpha-1}\left(e^{-\mu(s^{\alpha}+\lambda_{0})} +\sum_{n=1}^{\infty}u^{n}\sum_{k=1}^{n}\sum_{\Lambda_{n}^{k}}\prod_{j=1}^{n-k+1}\frac{(\lambda_{j-1}-\lambda_{j})^{k_{j}}}{k_{j}!} e^{-\mu(s^{\alpha}+\lambda_{0})}\mu^{k}\right)\mathrm{d}\mu \nonumber\\
		&=\int_{0}^{\infty}s^{\alpha-1}e^{-\mu s^{\alpha}}\left(e^{-\mu\lambda_{0}} +\sum_{n=1}^{\infty}u^{n}\sum_{k=1}^{n}\sum_{\Lambda_{n}^{k}}\prod_{j=1}^{n-k+1}\frac{(\lambda_{j-1}-\lambda_{j})^{k_{j}}}{k_{j}!} e^{-\mu\lambda_{0}}\mu^{k}\right)\mathrm{d}\mu \nonumber\\
		&=\int_{0}^{\infty}s^{\alpha-1}e^{-\mu s^{\alpha}}G_c(u,\mu)\mathrm{d}\mu,\ \ \text{(using\ (\ref{kjhgt21}))}\nonumber\\
		&=\int_{0}^{\infty}G_c(u,\mu)\int_{0}^{\infty}e^{-st}f_{T_{2\alpha}}(\mu,t)\mathrm{d}t \mathrm{d}\mu,\ \  \text{(using\ (\ref{lta}))}\nonumber\\
		&=\int_{0}^{\infty}e^{-st}\left(\int_{0}^{\infty}G_c(u,\mu)f_{T_{2\alpha}}(\mu,t)\mathrm{d}\mu\right)\mathrm{d}t \nonumber.
	\end{align}
By uniqueness of Laplace transform, we get
	\begin{equation*}
		G^{\alpha}_{c}(u,t)=\int_{0}^{\infty}G_c(u,\mu)f_{T_{2\alpha}}(\mu,t)\mathrm{d}\mu.
	\end{equation*}	
This completes the proof.
\end{proof}
	\begin{remark}
		For $\alpha=1/2$, the process $\{T_{2\alpha}(t)\}_{t>0}$ is equal in distribution to the reflecting Brownian motion $\{|B(t)|\}_{t>0}$ as the diffusion equation (\ref{diff}) reduces to the heat equation
		\begin{equation*}
			\begin{cases*}
				\frac{\partial}{\partial t}u(x,t)=\frac{\partial^{2}}{\partial x^{2}}u(x,t),\ \ x\in\mathbb{R},\ t>0,\\
				u(x,0)=\delta(x).
			\end{cases*}
		\end{equation*}
		So, $\mathcal{N}^{1/2}_{c}(t)$ coincides with CPP at a Brownian time, that is, $\mathcal{N}_{c}(|B(t)|)$, $t>0$.
	\end{remark}
\begin{remark}
	Let $\{H^{\alpha}(t)\}_{t>0}$, $0<\alpha\le 1$, be an inverse $\alpha$-stable subordinator. The density functions of $H^{\alpha}(t)$ and $T_{2\alpha}(t)$ coincides (see Meerschaert {\it et al.} (2011)). Thus, we have
	\begin{equation}\label{keyyekk}
	\mathcal{N}^{\alpha}_{c}(t)\overset{d}{=}\mathcal{N}_{c}(H^{\alpha}(t)),\ \ t>0,
	\end{equation}	
	where the inverse $\alpha$-stable subordinator is independent of the CPP.
\end{remark}
Let $\{\mathcal{H}^{\alpha}(t)\}_{t>0}$, $0<\alpha\le1$, be a random time process whose density function $f_{\mathcal{H}^{\alpha}(t)}(x,t)$, $x>0$, has the following Mellin transform (see Cahoy and Polito (2012), Eq. (3.12)):
\begin{equation}\label{mhts}
\int_{0}^{\infty}x^{\nu-1}f_{\mathcal{H}^{\alpha}(t)}(x,t)\mathrm{d}x=\frac{\Gamma(\nu)t^{\frac{\nu-1}{\alpha}}}{\Gamma(1-1/\alpha+\nu/\alpha)},\ \ t>0,\ \nu\in\mathbb{R}.
\end{equation}
\begin{proposition}
	Let the process $\{\mathcal{H}^{\alpha}(t)\}_{t>0}$, $0<\alpha\le 1$, be independent of the CFPP $\{\mathcal{N}^\alpha_{c}(t)\}_{t>0}$. Then,
	\begin{equation*}
	\mathcal{N}_{c}(t)\overset{d}{=}\mathcal{N}^{\alpha}_{c}(\mathcal{H}^{\alpha}(t)),\ \ t>0.
	\end{equation*}	
\end{proposition}
\begin{proof}
	Using (\ref{dssdew1}), we have
	\scriptsize
	\begin{align*}
	\int_{0}^{\infty}G^{\alpha}_{c}(u,s)&f_{\mathcal{H}^{\alpha}(t)}(s,t)\mathrm{d}s\\
	&=\int_{0}^{\infty}\sum_{n=0}^{\infty}u^np^{\alpha}_{c}(n,s)f_{\mathcal{H}^{\alpha}(t)}(s,t)\mathrm{d}s\\
	&=\int_{0}^{\infty}\left(E_{\alpha,1}(-\lambda_{0}s^{\alpha})+\sum_{n=1}^{\infty}u^{n}\sum_{k=1}^{n}\sum_{\Lambda_{n}^{k}}k!\prod_{j=1}^{n-k+1}\frac{(\lambda_{j-1}-\lambda_{j})^{k_{j}}}{k_{j}!} s^{k\alpha}E_{\alpha,k\alpha+1}^{k+1}(-\lambda_{0}s^{\alpha})\right)f_{\mathcal{H}^{\alpha}(t)}(s,t)\mathrm{d}s\\
	&=\int_{0}^{\infty}\Bigg(\frac{1}{2\pi i}\int_{c-i\infty}^{c+i\infty}\frac{\Gamma (z)\Gamma(1-z)}{\Gamma(1-\alpha z)}(\lambda_{0}s^{\alpha})^{-z}\mathrm{d}z\\
	&\ \ +\sum_{n=1}^{\infty}u^{n}\sum_{k=1}^{n}\sum_{\Lambda_{n}^{k}}\prod_{j=1}^{n-k+1}\frac{(\lambda_{j-1}-\lambda_{j})^{k_{j}}}{k_{j}!} 
	\frac{1}{2\pi i}\int_{c-i\infty}^{c+i\infty}\frac{\Gamma (z)\Gamma(k+1-z)}{\Gamma(\alpha(k-z)+1)}(\lambda_{0}s^{\alpha})^{-z}s^{k\alpha}\mathrm{d}z\Bigg)f_{\mathcal{H}^{\alpha}(t)}(s,t)\mathrm{d}s,\\
	&\hspace*{12cm} (\mathrm{using}\ (\ref{m3}))\\
	&=\frac{1}{2\pi i}\int_{c-i\infty}^{c+i\infty}\frac{\Gamma (z)\Gamma(1-z)\lambda_{0}^{-z}}{\Gamma(1-\alpha z)}\left(\int_{0}^{\infty}s^{-\alpha z}f_{\mathcal{H}^{\alpha}(t)}(s,t)\mathrm{d}s\right)\mathrm{d}z\\
	&\ \ +\sum_{n=1}^{\infty}u^{n}\sum_{k=1}^{n}\sum_{\Lambda_{n}^{k}}\prod_{j=1}^{n-k+1}\frac{(\lambda_{j-1}-\lambda_{j})^{k_{j}}}{k_{j}!}
	\frac{1}{2\pi i}\int_{c-i\infty}^{c+i\infty}\frac{\Gamma (z)\Gamma(k+1-z)\lambda_{0}^{-z}}{\Gamma(\alpha(k-z)+1)}\left(\int_{0}^{\infty}s^{k\alpha-\alpha z}f_{\mathcal{H}^{\alpha}(t)}(s,t)\mathrm{d}s\right)\mathrm{d}z\\
	&=\frac{1}{2\pi i}\int_{c-i\infty}^{c+i\infty}\Gamma (z)(\lambda_{0}t)^{-z}\mathrm{d}z +\sum_{n=1}^{\infty}u^{n}\sum_{k=1}^{n}\sum_{\Lambda_{n}^{k}}\prod_{j=1}^{n-k+1}\frac{(\lambda_{j-1}-\lambda_{j})^{k_{j}}}{k_{j}!}
	\frac{t^{k}}{2\pi i}\int_{c-i\infty}^{c+i\infty}\Gamma (z)(\lambda_{0}t)^{-z}\mathrm{d}z,\\
	&\hspace*{12cm} (\mathrm{using}\  (\ref{mhts}))\\
	&=e^{-t\lambda_{0}} +\sum_{n=1}^{\infty}u^{n}\sum_{k=1}^{n}\sum_{\Lambda_{n}^{k}}\prod_{j=1}^{n-k+1}\frac{(\lambda_{j-1}-\lambda_{j})^{k_{j}}}{k_{j}!} t^{k}e^{-t\lambda_{0}},\ \ (\mathrm{using}\ (\ref{me}))\\
	&
	=G_c(u,t),
	\end{align*}
	\normalsize
	where in the last step we have used (\ref{kjhgt21}).
\end{proof}
\section{The dependence structer for CFPP and its increments}\label{Section5}
In this section, we show that the CFPP has LRD property whereas its increments exhibits the SRD property. 

For a non-stationary stochastic process $\{X(t)\}_{t\geq0}$ the LRD and SRD properties are defined as follows (see D'Ovidio and Nane (2014), Maheshwari and Vellaisamy (2016)):
\begin{definition}
	Let $s>0$ be fixed and $\{X(t)\}_{t\ge0}$ be a stochastic process whose correlation function satisfies
	\begin{equation}\label{lrd}
	\operatorname{Corr}(X(s),X(t))\sim c(s)t^{-\gamma},\  \text{as}\ t\rightarrow\infty,
	\end{equation}
for some $c(s)>0$. The process $\{X(t)\}_{t\ge0}$ has the LRD property if $\gamma\in(0,1)$ and the SRD property if $\gamma\in(1,2)$.
\end{definition}

First, we obtain the covariance of CFPP. Using Theorem 2.1 of Leonenko {\it et al.} (2014) and the subordination result (\ref{keyyekk}), the covariance of CFPP can be obtained as follows:
\begin{align}\label{covfrd11}
\operatorname{Cov}\left(\mathcal{N}^{\alpha}_{c}(s),\mathcal{N}^{\alpha}_{c}(t)\right)&=\operatorname{Cov}\left(\mathcal{N}_{c}(H^{\alpha}(s)),\mathcal{N}_{c}(H^{\alpha}(t))\right)\nonumber\\
&=\operatorname{Var}\left(\mathcal{N}_{c}(1)\right)\mathbb{E}(H^{\alpha}(\min\{s,t\}))+\left(\mathbb{E}(\mathcal{N}_{c}(1))\right)^{2}\operatorname{Cov}\left(H^{\alpha}(s),H^{\alpha}(t)\right),
\end{align}
where we used Theorem \ref{thwe22} and the fact that the inverse stable subordinator $\{H^{\alpha}(t)\}_{t\ge0}$ is a non-decreasing process. 

On using Theorem 2.1 of Leonenko {\it et al.} (2014), the mean and variance of CFPP can alternatively be obtained as follows:
\begin{equation*}
\mathbb{E}\left(\mathcal{N}^{\alpha}_{c}(t)\right)=\mathbb{E}\left(\mathcal{N}_{c}(1)\right)\mathbb{E}\left(H^{\alpha}(t)\right)
\end{equation*}
and
\begin{equation*}
\operatorname{ Var}\left(\mathcal{N}^{\alpha}_{c}(t)\right)=\left(\mathbb{E}\left(\mathcal{N}_{c}(1)\right)\right)^{2}\operatorname{ Var}\left(H^{\alpha}(t)\right)+\operatorname{ Var}\left(\mathcal{N}_{c}(1)\right)\mathbb{E}\left(H^{\alpha}(t)\right),
\end{equation*}
where the mean and variance of inverse $\alpha$-stable subordinator are given by (see Leonenko {\it et al.} (2014), Eq. (8) and Eq. (11))
	\begin{equation*}
	\mathbb{E}\left(H^{\alpha}(t)\right)=\frac{t^{\alpha}}{\Gamma(\alpha+1)}
	\end{equation*}
	and 	
	\begin{equation}\label{xswe33}
	\operatorname{ Var}\left(H^{\alpha}(t)\right)=t^{2\alpha}\left(\frac{2}{\Gamma(2\alpha+1)}-\frac{1}{\Gamma^{2}(\alpha+1)}\right),
	\end{equation}
	respectively.
	\begin{remark}
		From (\ref{wswee11}), (\ref{var}) and (\ref{xswe33}), we get
		\begin{equation*}
		\operatorname{ Var}\left(\mathcal{N}^{\alpha}_{c}(t)\right)-\mathbb{E}\left(\mathcal{N}^{\alpha}_{c}(t)\right)=\frac{2t^{\alpha}\sum_{j=1}^{\infty}j\lambda_{j}}{\Gamma(\alpha+1)}+\left(\sum_{j=0}^{\infty}\lambda_{j}\right)^{2}	\operatorname{ Var}\left(H^{\alpha}(t)\right).
		\end{equation*}
		Thus, the CFPP exhibits overdispersion as $ \operatorname{ Var}\left(\mathcal{N}^{\alpha}_{c}(t)\right)-\mathbb{E}\left(\mathcal{N}^{\alpha}_{c}(t)\right)>0$ for $t>0$.
	\end{remark}
Let
\begin{equation*}
R=\frac{1}{\Gamma(\alpha+1)}\sum_{j=0}^{\infty}\lambda_{j},\ \ S=\left(\frac{2}{\Gamma(2\alpha+1)}-\frac{1}{\Gamma^2(\alpha+1)}\right)\left(\sum_{j=0}^{\infty}\lambda_{j}\right)^{2}
\end{equation*}
and
\begin{equation*}
T=\frac{1}{\Gamma(\alpha+1)}\left(\sum_{j=0}^{\infty}\lambda_{j}+2\sum_{j=1}^{\infty}j\lambda_{j}\right).
\end{equation*}
For $0<s\le t$ in (\ref{covfrd11}), we get
\begin{equation}\label{qazxsa22}
\operatorname{Cov}\left(\mathcal{N}^{\alpha}_{c}(s),\mathcal{N}^{\alpha}_{c}(t)\right)=Ts^{\alpha}+\left(\sum_{j=0}^{\infty}\lambda_{j}\right)^{2}\operatorname{Cov}\left(H^{\alpha}(s),H^{\alpha}(t)\right).
\end{equation}
For large $t$, we use the following result due to Leonenko {\it et al.} (2014):
\begin{equation*}\label{asi}
\operatorname{Cov}\left(H^{\alpha}(s),H^{\alpha}(t)\right)\sim\frac{s^{2\alpha}}{\Gamma(2\alpha+1)}.
\end{equation*}
in (\ref{qazxsa22}), to obtain
\begin{equation}\label{wdseq16}
\operatorname{Cov}\left(\mathcal{N}^{\alpha}_{c}(s),\mathcal{N}^{\alpha}_{c}(t)\right)\sim Ts^{\alpha}+\frac{\left(\sum_{j=0}^{\infty}\lambda_{j}\right)^{2}s^{2\alpha}}{\Gamma(2\alpha+1)}\ \ \mathrm{as}\ \ t\to\infty.
\end{equation}
We now show that the CFPP has LRD property.
 	\begin{theorem}
		The CFPP exhibits the LRD property.
	\end{theorem}
	\begin{proof}
		Using (\ref{var}) and (\ref{wdseq16}), we get the following for fixed $s>0$ and large $t$:
		\begin{align*}
		\operatorname{Corr}\left(\mathcal{N}^{\alpha}_{c}(s),\mathcal{N}^{\alpha}_{c}(t)\right)&=\frac{\operatorname{Cov}\left(\mathcal{N}^{\alpha}_{c}(s),\mathcal{N}^{\alpha}_{c}(t)\right)}{\sqrt{\operatorname{ Var}\left(\mathcal{N}^{\alpha}_{c}(s)\right)}\sqrt{\operatorname{ Var}\left(\mathcal{N}^{\alpha}_{c}(t)\right)}}\\
		&\sim\frac{\Gamma(2\alpha+1)Ts^{\alpha}+\left(\sum_{j=0}^{\infty}\lambda_{j}\right)^{2}s^{2\alpha}}{\Gamma(2\alpha+1)\sqrt{\operatorname{ Var}\left(\mathcal{N}^{\alpha}_{c}(s)\right)}\sqrt{St^{2\alpha}+Tt^{\alpha}}}\\
		&\sim c_0(s)t^{-\alpha},
		\end{align*}
where
\begin{equation*}
c_0(s)=\frac{\Gamma(2\alpha+1)Ts^{\alpha}+\left(\sum_{j=0}^{\infty}\lambda_{j}\right)^{2}s^{2\alpha}}{\Gamma(2\alpha+1)\sqrt{\operatorname{ Var}\left(\mathcal{N}^{\alpha}_{c}(s)\right)}\sqrt{S}}.
\end{equation*}
As $0<\alpha<1$, the result follows.
\end{proof}
For a fixed $\delta>0$, we define the convoluted fractional Poissonian noise (CFPN), denoted by $\{Z^{\alpha}_{c,\delta}(t)\}_{t\ge0}$, as the increment process of CFPP, that is, 
	\begin{equation}\label{qlq1}
	Z^{\alpha}_{c,\delta}(t)\coloneqq\mathcal{N}^{\alpha}_{c}(t+\delta)-\mathcal{N}^{\alpha}_{c}(t).
	\end{equation}
Next, we show that the CFPN exhibits the SRD property.
	\begin{theorem}\label{varbgfff}
		The CFPN has the SRD property.
	\end{theorem}
	\begin{proof}
		Let $s\ge0$ be fixed such that $0\le s+\delta\le t$. We have,
		\begin{align}\label{covz}
		\operatorname{Cov}(Z^{\alpha}_{c,\delta}(s),Z^{\alpha}_{c,\delta}(t))&=\operatorname{Cov}\left(\mathcal{N}^{\alpha}_{c}(s+\delta)-\mathcal{N}^{\alpha}_{c}(s),\mathcal{N}^{\alpha}_{c}(t+\delta)-\mathcal{N}^{\alpha}_{c}(t)\right)\nonumber\\
		&=\operatorname{Cov}\left(\mathcal{N}^{\alpha}_{c}(s+\delta),\mathcal{N}^{\alpha}_{c}(t+\delta)\right)+\operatorname{Cov}\left(\mathcal{N}^{\alpha}_{c}(s),\mathcal{N}^{\alpha}_{c}(t)\right)\nonumber\\
		&\ \ \ \  -\operatorname{Cov}\left(\mathcal{N}^{\alpha}_{c}(s+\delta),\mathcal{N}^{\alpha}_{c}(t)\right)-\operatorname{Cov}\left(\mathcal{N}^{\alpha}_{c}(s),\mathcal{N}^{\alpha}_{c}(t+\delta)\right).
		\end{align}
		Leonenko {\it et al.} (2014) obtained the following expression for the covariance of inverse $\alpha$-stable subordinators:
		\begin{equation}\label{mhjgf44}
		\operatorname{Cov}\left(H^{\alpha}(s),H^{\alpha}(t)\right)=\frac{1}{\Gamma^2(\alpha+1)}\left( \alpha s^{2\alpha}B(\alpha,\alpha+1)+F(\alpha;s,t)\right),
		\end{equation}
		where $F(\alpha;s,t)=\alpha t^{2\alpha}B(\alpha,\alpha+1;s/t)-(ts)^{\alpha}$. Here, $B(\alpha,\alpha+1)$ is the beta function whereas $B(\alpha,\alpha+1;s/t)$ is the incomplete beta function. 
		
		In (\ref{qazxsa22}), we use the following asymptotic result (see Maheshwari and Vellaisamy (2016), Eq. (8)):
		\begin{equation*}
		F(\alpha;s,t)\sim \frac{-\alpha^{2}}{(\alpha+1)}\frac{s^{\alpha+1}}{t^{1-\alpha}},\ \ \mathrm{as}\ \ t\to\infty,
		\end{equation*}
		to obtain 
		\begin{equation}\label{covzt}
		\operatorname{Cov}\left(\mathcal{N}^{\alpha}_{c}(s),\mathcal{N}^{\alpha}_{c}(t)\right)\sim Ts^{\alpha}+ R^{2}\left(\alpha s^{2\alpha}B(\alpha,\alpha+1)-\frac{\alpha^{2}}{(\alpha+1)}\frac{s^{\alpha+1}}{t^{1-\alpha}}\right) \ \ \mathrm{as}\  t\to\infty.
		\end{equation}
		From (\ref{covz}) and (\ref{covzt}), we get the following for large $t$:
		\begin{align}\label{covzi}
		\operatorname{Cov}(Z^{\alpha}_{c,\delta}(s),Z^{\alpha}_{c,\delta}(t))&\sim \frac{R^{2}\alpha^{2}}{\alpha+1}\left(\frac{s^{\alpha+1}}{(t+\delta)^{1-\alpha}}+\frac{(s+\delta)^{\alpha+1}}{t^{1-\alpha}}-\frac{s^{\alpha+1}}{t^{1-\alpha}}-\frac{(s+\delta)^{\alpha+1}}{(t+\delta)^{1-\alpha}}\right)\nonumber\\
		&=\frac{R^{2}\alpha^{2}}{\alpha+1}\left((t+\delta)^{\alpha-1}-t^{\alpha-1}\right)\left(s^{\alpha+1}-(s+\delta)^{\alpha+1}\right)\nonumber\\
		&\sim\frac{\alpha^{2}\delta(1-\alpha)}{\alpha+1}\left((s+\delta)^{\alpha+1}-s^{\alpha+1}\right)R^{2}t^{\alpha-2}.
		\end{align}
		Now,
		\begin{equation}\label{rfcdee1}
		\operatorname{Var}(Z^{\alpha}_{c,\delta}(t))=\operatorname{Var}(\mathcal{N}^{\alpha}_{c}(t+\delta))+\operatorname{Var}(\mathcal{N}^{\alpha}_{c}(t))-2\operatorname{Cov}\left(\mathcal{N}^{\alpha}_{c}(t),\mathcal{N}^{\alpha}_{c}(t+\delta)\right).
		\end{equation}
		From (\ref{qazxsa22}) and (\ref{mhjgf44}), we have
		\begin{equation}\label{sxeww12}
		\operatorname{Cov}\left(\mathcal{N}^{\alpha}_{c}(t),\mathcal{N}^{\alpha}_{c}(t+\delta)\right)=Tt^{\alpha}+R^{2}\left(\alpha t^{2\alpha}B(\alpha,\alpha+1)+F(\alpha;t,t+\delta)\right),
		\end{equation}
		where $F(\alpha;t,t+\delta)=\alpha (t+\delta)^{2\alpha}B(\alpha,\alpha+1,t/t+\delta)-(t(t+\delta))^{\alpha}$.	
		
		For large $t$, we have
		\begin{equation*}
	B(\alpha,\alpha+1,t/t+\delta)\sim B(\alpha,\alpha+1)=\frac{\Gamma(\alpha)\Gamma(\alpha+1)}{\Gamma(2\alpha+1)}.
		\end{equation*}
		Substituting (\ref{var}) and (\ref{sxeww12}) in (\ref{rfcdee1}), we get
		\begin{align}\label{varzi}
		\operatorname{Var}(Z^{\alpha}_{c,\delta}(t))&\sim (S-2R^{2}\alpha B(\alpha,\alpha+1))t^{2\alpha}+(S-2R^{2}\alpha B(\alpha,\alpha+1))(t+\delta)^{2\alpha}\nonumber\\
		&\hspace*{4cm}+T\left((t+\delta)^{\alpha}-t^{\alpha}\right)+2R^{2}(t(t+\delta))^{\alpha}\nonumber\\
		&= Tt^{\alpha}\left(\left(1+\frac{\delta}{t}\right)^{\alpha}-1\right)-R^{2}t^{2\alpha}\left(\left(1+\frac{\delta}{t}\right)^{\alpha}-1\right)^2\nonumber\\
		&\sim T\alpha\delta t^{\alpha-1}-R^{2}\alpha^{2}\delta^{2}t^{2\alpha-2}\nonumber\\
		&\sim \alpha\delta Tt^{\alpha-1},\ \ \mathrm{as}\ \ t\to\infty.
		\end{align}
		From (\ref{covzi}) and (\ref{varzi}), we have
		\begin{align*}
		\operatorname{Corr}(Z^{\alpha}_{c,\delta}(s),Z^{\alpha}_{c,\delta}(t))&=\dfrac{\operatorname{Cov}\left(Z^{\alpha}_{c,\delta}(s),Z^{\alpha}_{c,\delta}(t)\right)}{\sqrt{\operatorname{Var}(Z^{\alpha}_{c,\delta}(s))}\sqrt{\operatorname{Var}(Z^{\alpha}_{c,\delta}(t))}}\\
		&\sim \frac{\alpha^{2}\delta(1-\alpha)\left((s+\delta)^{\alpha+1}-s^{\alpha+1}\right)R^{2}t^{\alpha-2}}{(\alpha+1)\sqrt{\operatorname{Var}(Z^{\alpha}_{c,\delta}(s))}\sqrt{\alpha\delta T t^{\alpha-1}}}\\
		&=c_1(s)t^{-(3-\alpha)/2},\ \ \mathrm{as}\ t\rightarrow\infty.
		\end{align*}
	where
	\begin{equation*}
	c_1(s)=\frac{\alpha^{2}\delta(1-\alpha)\left((s+\delta)^{\alpha+1}-s^{\alpha+1}\right)R^{2}}{(\alpha+1)\sqrt{\operatorname{Var}(Z^{\alpha}_{c,\delta}(s))}\sqrt{\alpha\delta T}}.
	\end{equation*}
		Thus, the CFPN exhibits the SRD property as $1<(3-\alpha)/2<3/2$.
	\end{proof}

\printbibliography

\begin{thebibliography}{1}
\bibitem{Aletti2018}
Aletti, G., Leonenko, N. N. and Merzbach, E. (2018). Fractional Poisson fields and martingales, {\it J. Stat. Phys.} 170(4), 700-730.
\bibitem{Beghin2012}
Beghin, L. (2012). Random-time processes governed by differential equations of fractional distributed order. {\it Chaos Solitons Fractals}. 45(11), 1314-1327. 
\bibitem{Beghin2009}
Beghin, L. and Orsingher, E. (2009). Fractional Poisson processes and related planar random motions. {\it Electron. J. Probab.} 14(61), 1790-1827.
\bibitem{Beghin2018}
Beghin, L. and Vellaisamy, P. (2018).  Space-fractional versions of the negative binomial and Polya-type processes. {\it  Methodol. Comput. Appl. Probab.} 20(2), 463–485.
\bibitem{Biard2014}
Biard, R. and Saussereau, B. (2014). Fractional Poisson process: long-range dependence and applications in ruin theory. {\em J. Appl. Probab.} {51}(3), 727-740.
\bibitem{Cahoy2012}
Cahoy, D. O. and Polito, F. (2012).  Simulation and estimation for the fractional Yule process. {\it Methodol. Comput. Appl. Probab.} 14(2), 383-403.
\bibitem{Comtet1974}
Comtet, L. (1974). Advanced Combinatorics: The Art of Finite and Infinite Expansions, D. Reidel Publishing Co., Dordrecht.
\bibitem{Damelin2012}		
Damelin, S. B. and  Miller, W. Jr. (2012). The Mathematics of Signal Processing, Cambridge University Press, Cambridge.
\bibitem{Crescenzo2016}	
Di Crescenzo, A., Martinucci, B. and Meoli, A. (2016). A fractional counting process and its connection with the Poisson process. {\it ALEA Lat. Am. J. Probab. Math. Stat.} 13(1), 291–307.
\bibitem{Ovidio2014}
D'Ovidio, M. and Nane, E. (2014). Time dependent random fields on spherical non-homogeneous surfaces. {\it Stochastic Process. Appl.} 124(6), 2098-2131.
\bibitem{Garra2015}
Garra, R., Orsingher, E. and Polito, F. (2015). State-dependent fractional point processes. {\em J. Appl. Probab.} 52(1), 18-36.
\bibitem{Johnson2002}
Johnson, W. P. (2002). The curious history of Fa\`a di Bruno's formula. {\it Am. Math. Mon. } 109(3), 217-234.
\bibitem{Kataria2020}
Kataria, K. K. and Khandakar, M. (2020). On the long-range dependence of mixed fractional Poisson process. {\em J. Theoret. Probab}. To appear.
\bibitem{Kataria2017a}
Kataria, K. K. and Vellaisamy, P. (2017a). Saigo space-time fractional Poisson process via Adomian decomposition method. {\em Statist. Probab. Lett.} 129, 69-80.
\bibitem{Kataria2017b}
Kataria, K. K. and Vellaisamy, P. (2017b). Correlation between Adomian and partial exponential Bell polynomials. {\it C. R. Math. Acad. Sci. Paris}. 355(9), 929-936. 
\bibitem{Kataria2019}
Kataria, K. K. and Vellaisamy, P. (2019). On the convolution of Mittag-Leffler distributions and its
applications to fractional point processes. {\it Stoch. Anal. Appl.} 37(1), 115-122.
\bibitem{Kilbas2006}
Kilbas, A. A., Srivastava, H. N. and Trujillo J. J. (2006). Theory and Applications of Fractional Differential Equations. North-Holland
mathematics studies, vol. 204. Elsevier, Amsterdam.	
\bibitem{Laskin2003}	
Laskin, N., (2003). Fractional Poisson process. {\it Commun. Nonlinear Sci. Numer. Simul.} 8(3-4), 201-213.
\bibitem{Leonenko2014}
Leonenko, N. N., Meerschaert, M. M., Schilling, R. L. and Sikorskii, A. (2014). Correlation structure of time-changed  L\'evy processes. {\it Commun. Appl. Ind. Math.} 6(1), e-483, 22.
\bibitem{Maheshwari2016}
Maheshwari, A. and  Vellaisamy, P. (2016). On the long-range dependence of fractional Poisson and negative binomial processes. {\it J. Appl. Probab.} 53(4), 989-1000.
\bibitem{Mathai2008}
Mathai, A. M. and Haubold, H. J. (2008). Special Functions for Applied Scientists. New York, Springer.	
\bibitem{Meerschaert2011}	
Meerschaert, M. M., Nane, E. and Vellaisamy, P. (2011). The fractional Poisson process and the inverse stable subordinator. {\it Electron. J. Probab.} 16(59), 1600-1620.
\bibitem{Oliveira2016}	
de Oliveira, D. S., Capelas de Oliveira, E. and Deif, S. (2016). On a sum with a three-parameter Mittag-Leffler function. {\it Integral Transforms Spec. Funct.} 27(8), 639–652. 
\bibitem{Orsingher2004}
Orsingher, E. and  Beghin, L. (2004). Time-fractional telegraph  equations and telegraph processes with Brownian time. {\it Probab. Theory Related Fields}, 128(1), 141-160.
\bibitem{Orsingher2010}
Orsingher, E. and Polito, F. (2010). Fractional pure birth processes. {\it Bernoulli}. 16(3), 858-881.
\bibitem{Orsingher2012}
Orsingher, E. and Polito, F. (2012). The space-fractional Poisson process. {\em Statist. Probab. Lett.} 82(4), 852-858.
\bibitem{Orsingher2015}
Orsingher, E. and Toaldo, B. (2015). Counting processes with Bern\v stein intertimes and random jumps. {\em J. Appl. Probab.} 52(4), 1028-1044.
\bibitem{Paris2001}	
Paris, R. B. and Kaminski, D. (2001). Asymptotics and Mellin-Barnes Integrals. Cambridge University Press, Cambridge.
\bibitem{Polito2016}
Polito, F. and Scalas, E. (2016). A generalization of the space-fractional Poisson process and its connection to some L\'evy processes. {\em Electron. Commun. Probab.} 21, 1-14.
\end{thebibliography}
\end{document}